\documentclass[a4paper,11pt,english]{smfart}
\usepackage{amsfonts}
\usepackage{amsmath,amsthm}
\usepackage{latexsym}
\usepackage{array}
\usepackage{amssymb}
\usepackage{graphicx}
\usepackage{enumerate}
\usepackage{verbatim}
\usepackage{float}
\usepackage[english]{babel}

\usepackage{color}
\definecolor{marin}{rgb}   {0.,   0.3,   0.7} 
\definecolor{rouge}{rgb}   {0.8,   0.,   0.} 
\definecolor{sepia}{rgb}   {0.8,   0.5,   0.} 
\usepackage[colorlinks,citecolor=marin,linkcolor=rouge,
            bookmarksopen,
            bookmarksnumbered
           ]{hyperref}

\usepackage{subcaption}

\theoremstyle{plain} 
\newtheorem{theorem}{Theorem}[section]
\newtheorem{lemma}[theorem]{Lemma}

\newtheorem{corollary}[theorem]{Corollary} 
 \theoremstyle{remark}
\newtheorem{remark}[theorem]{Remark}
\newtheorem{example}[theorem]{Example}
 
\newtheorem{ass}[theorem]{Assumption}

\newcommand {\aplt} {\ {\raise-.5ex\hbox{$\buildrel<\over\sim$}}\ } 
\newcommand {\gplt} {\ {\raise-.5ex\hbox{$\buildrel>\over\sim$}}\ }

\makeatletter
\def\@makefnmark{\hbox{$\m@th^{\@thefnmark}$}}
\makeatother
\begin{document}
\author{Martina Hofmanov\'a}
\address{Institute of mathematics, Technical University Berlin,
Strasse des 17. Juni 136, 10623 Berlin, Germany}
\email{hofmanov@math.tu-berlin.de}

\author{Katharina Schratz}
\address{Fakult\"{a}t f\"{u}r Mathematik, Karlsruhe Institute of Technology,
Englerstr. 2, 76131 Karlsruhe, Germany}
\email{katharina.schratz@kit.edu}

\begin{abstract}
We introduce an exponential-type time-integrator for the KdV equation and prove its first-order  convergence in $H^1$ for initial data in $H^3$.  Furthermore, we outline the generalization of the presented technique to a second-order method. 
\end{abstract}

\keywords{KdV equation -- exponential-type time integrator -- convergence}

\title{An exponential-type integrator for the KdV equation}
\maketitle
We consider the Korteweg-de Vries (KdV) equation
\begin{equation}\label{eq:kdv}
\partial_t u(t,x) + \partial_x^3 u(t,x) = \frac{1}{2} \partial_x (u(t,x))^2 , \quad u(0,x) = u_0(x),\quad t \in \mathbb{R}, \quad x \in \mathbb{T}= [-\pi,\pi],
\end{equation}
where for practical implementation issues we impose periodic boundary conditions. For local-wellposedness results of the periodic KdV equation in low regularity spaces we refer to \cite{Bour93,Gub11,Tao06}.

In the context of the numerical time integration of (non)linear partial differential equations splitting methods as well as exponential integrators contribute  attractive classes of integration methods. We refer to \cite{HLW,HochOst10,HLRS10,McLacQ02} for an extensive overview, and in particular to \cite{Faou12,Gau11,Lubich08} for the analysis of splitting methods for Schr\"odinger(-Poisson) equations. In recent years, splitting as well as exponential integration schemes (including Lawson type Runge-Kutta methods \cite{Law67}) have also gained a lot of attention in the context of the numerical integration of the KdV equation, see for instance \cite{HLR12,HKRT12,HKR99,KT05,Klein06,T74} and the references therein. We also refer to~\cite{OE15} for a splitting approach for the Kadomtsev-Petviashvili equation.

In particular, a distinguished convergence result was obtained in \cite{HKRT12,HLR12}. In the latter it was proven that the Strang splitting, where the right-hand side of the KdV equation is split into the linear and Burgers part, respectively, is second-order convergent in $H^r$ for initial data in $H^{r+5}$ for $r \geq 1$ assuming that the Burgers part is solved exactly. 

Here we derive a first-order exponential-type time-integrator for the KdV equation \eqref{eq:kdv} based  on Duhamel's formula
\begin{equation}\label{kdv}
u(t) = \mathrm{e}^{-\partial_{x}^3 t}u_0 + \frac{1}{2}\int_0^t \mathrm{e}^{-\partial_{x}^3 (t-s)} \partial_{x} (u(s))^2 \mathrm{d} s
\end{equation}
looking at the ``twisted variable" $v (t) = \mathrm{e}^{\partial_{x}^3 t} u(t)$. This idea of ``twisting" the variable is widely used in the analysis of partial differential equations in low regularity spaces (see, for instance \cite{Bour93,Gub11,Tao06} for the periodic KdV equation) and also well known in the context of numerical analysis, see \cite{Law67} for the introduction of Lawson type Runge-Kutta methods. However, instead of approximating the appearing integral with a Runge-Kutta method (see for instance \cite{KT05}) we use the key relation
\begin{align}\label{eq:key}
k_1^3+k_2^3 - (k_1+k_2)^3 = - 3 (k_1+k_2) k_1k_2
\end{align}
which allows us to overcome the \emph{loss of derivative} by integrating the \emph{stiff parts} (i.e., the terms involving $\partial_x^3$) exactly. The derived exponential-type integrator is unconditionally  stable and we will in particular show its first-order convergence in $H^1$ for initial data in $H^3$. A key tool in our convergence analysis is a variant of \cite[Lemma 3.1]{HLR12}.

 The presented technique can  be generalized to higher-order methods. We outline the construction of a second-order exponential-type integrator in Remark \ref{rem:2ord}. \\

\textbf{Notation: } In the following we will denote the Fourier expansion of some function $f \in L^2(\mathbb{T})$ by $
f(x) = \sum_{k \in \mathbb{Z}} \hat{f}_k \mathrm{e}^{i k x}$. Furthermore, we will use the notation
\begin{equation}\label{nota}
(\partial_x^{-1})_k : = 
\left\{
\begin{array}{ll}
(ik)^{-1} &\mbox{if} \quad k \neq 0\\
0 & \mbox{if}\quad k = 0 
\end{array}
\right.,  \quad \text{ i.e.,} \qquad\partial_x^{-1} f(x) =\sum_{\substack{k\in \mathbb{Z}\\k \neq 0} } (ik)^{-1} \hat{f}_k \mathrm{e}^{ikx}.
\end{equation}

\section{An exponential-type integrator}
To illustrate the idea we first consider initial values with zero mean. In Remark \ref{rem:0mn0} we point out the generalization to general initial values.

\begin{ass}\label{as0}
Assume that the zero-mode of the initial value is zero, i.e., $\hat{u}_0(0) =(2\pi)^{-1}\int_{\mathbb{T}}u(0,x) \mathrm{d} x= 0$. Note that the conservation of mass then implies that $\hat{u}_0(t) = 0$.
\end{ass}
We will derive a scheme for the ``twisted" variable $v (t) = \mathrm{e}^{\partial_{x}^3 t} u(t)$. With this transformation at hand the equation in $v$ reads
\begin{equation}\label{kdvT}
v(t) = v_0 + \frac{1}{2}\int_0^t \mathrm{e}^{s \partial_{x}^3} \partial_x (\mathrm{e}^{-\partial_{x}^3 s} v(s))^2 \mathrm{d}s
\end{equation}
such that
\begin{equation}\label{To}
v(t_n+\tau) =  v(t_n) + \frac{1}{2}\int_0^\tau \mathrm{e}^{(t_n+s) \partial_{x}^3} \partial_x\left (\mathrm{e}^{-\partial_{x}^3 (t_n+s)} v(t_n+s)\right)^2 \mathrm{d}s.
\end{equation}
For a small time-step $\tau$ we iterate Duhamel's formula \eqref{kdvT} and approximate the exact solution~\eqref{To} as follows
\begin{equation}
\begin{aligned}
v(t_n+\tau ) 
& \approx v(t_n) + \frac{1}{2}\int_0^\tau \mathrm{e}^{(t_n+s) \partial_{x}^3} \partial_x \left(\mathrm{e}^{-\partial_{x}^3 (t_n+s)} v(t_n)\right)^2 \mathrm{d}s\label{app}.
\end{aligned}
\end{equation}
The key relation \eqref{eq:key} now allows us the following integration technique (cf. \cite{Bour93,Gub11,Tao06}): We have
\begin{equation}\label{int}
\begin{aligned}
& \int_0^\tau \mathrm{e}^{(t_n+s) \partial_{x}^3} \partial_x (\mathrm{e}^{-\partial_{x}^3 (t_n+s)} v(t_n))^2 \mathrm{d}s\\
& = \sum_{k_1,k_2} \int_0^\tau \mathrm{e}^{-i(t_n+s) \big((k_1+k_2)^3-k_1^3-k_2^3\big)} i (k_1+k_2) \hat v_{k_1}(t_n) \hat v_{k_2}(t_n) \mathrm{e}^{i (k_1+k_2) x}\mathrm{d}s\\
& =  \sum_{k_1,k_2} \frac{ \mathrm{e}^{-i(t_n+\tau) \big((k_1+k_2)^3-k_1^3-k_2^3\big)}-  \mathrm{e}^{-it_n \big((k_1+k_2)^3-k_1^3-k_2^3\big)}}{-i\big( (k_1+k_2)^3-k_1^3-k_2^3\big)} i(k_1+k_2) \hat v_{k_1}(t_n) \hat v_{k_2}(t_n)\mathrm{e}^{i (k_1+k_2) x}\\
& =   \sum_{k_1,k_2} \left( \mathrm{e}^{-i(t_n+\tau) \big((k_1+k_2)^3-k_1^3-k_2^3\big)}-  \mathrm{e}^{-it_n \big((k_1+k_2)^3-k_1^3-k_2^3\big)}\right) \frac{1}{-3 k_1k_2} \hat v_{k_1}(t_n) \hat v_{k_2}(t_n)\mathrm{e}^{i (k_1+k_2) x}\\
& =  \frac{1}{3}\mathrm{e}^{\partial_{x}^3 (t_n+\tau)} \left( \mathrm{e}^{-\partial_{x}^3(t_n+\tau)} \partial_x^{-1}v(t_n)\right)^2 - \frac{1}{3}\mathrm{e}^{\partial_{x}^3 t_n} \left( \mathrm{e}^{-\partial_{x}^3t_n} \partial_x^{-1}v(t_n)\right)^2.
\end{aligned}
\end{equation}
Together with the approximation in \eqref{app} this yields that
\begin{align}\label{scheme}
v^{n+1}&= v^n + \frac{1}{6}\mathrm{e}^{\partial_{x}^3 (t_n+\tau)} \left( \mathrm{e}^{-\partial_{x}^3(t_n+\tau)} \partial_x^{-1}v^n\right)^2 - \frac{1}{6}\mathrm{e}^{\partial_{x}^3 t_n} \left( \mathrm{e}^{-\partial_{x}^3t_n} \partial_x^{-1}v^n\right)^2,
\end{align}
where $\partial_x^{-1}$ is defined in \eqref{nota} and by construction $\hat{v}_0^{n+1} = 0$, see Remark \ref{R:ZM-scheme}  below.
\begin{remark}\label{R:ZM-scheme} The zero-mode is preserved by the scheme \eqref{scheme} as the key relation \eqref{eq:key} implies that
\begin{equation*}
\begin{aligned}
\hat{v}_0^{n+1} & =
\hat{v}_0^n - \frac{1}{6}   \sum_{k_1+k_2=0}\left( \mathrm{e}^{-i(t_n+\tau)3 (k_1+k_2) k_1 k_2}-  \mathrm{e}^{-it_n 3 (k_1+k_2) k_1 k_2}\right) \frac{1}{ k_1k_2} \hat v_{k_1}^n \hat v_{k_2}^n = \hat{v}_0^n.
\end{aligned}
\end{equation*}
\end{remark}
In order to obtain an approximation to the original solution $u(t_n)$ of the KdV equation~\eqref{eq:kdv} at time $t_n = n \tau$ we then ``twist" the variable back again by setting $u^{n}= \mathrm{e}^{- \partial_x^3 t_n} v^{n}$. This yields the following exponential-type integrator for the KdV equation \eqref{eq:kdv}\begin{equation}\label{schemeU}
u^{n+1}= \mathrm{e}^{-\tau \partial_x^3}u^n + \frac{1}{6} \Big( \mathrm{e}^{- \tau \partial_x^3} \partial_x^{-1} u^n\Big)^2 - \frac{1}{6} \mathrm{e}^{- \tau \partial_x^3}\Big( \partial_x^{-1}u^n\Big)^2,
\end{equation}
where $\partial_x^{-1}$ is defined in \eqref{nota} and  $ \hat{u}_0^{n+1}= 0$ thanks to Remark \ref{R:ZM-scheme}. For sufficiently smooth solutions the semi-discrete scheme \eqref{schemeU} is first-order convergent, see Corollary~\ref{cor:ucon} below for the precise convergence result.

\begin{remark}\label{rem:0mn0}
If $\hat{u}_0(0) = \alpha \neq 0$ we set $\tilde{u}:= u - \alpha$ and look at the modified KdV equation in $\tilde{u}$, i.e.,
\begin{equation}\label{modkdv}
\partial_t \tilde{u}+ \partial_{x}^3 \tilde{u}  = \alpha \partial_x \tilde{u} + \frac{1}{2} \partial_x (\tilde{u})^2 .
\end{equation}
Note that the solution $\tilde{u}(t)$  of the modified KdV equation \eqref{modkdv} satisfies $\hat{\tilde{u}}_0(t) = 0$ for all $t$ as by the conservation of mass we have that $\hat{u}_0(t) \equiv \hat{u}_0(0) = \alpha$.  Thus, we can proceed as above: We look at the twisted variable $\tilde{v}(t) = \mathrm{e}^{(\partial_x^3-\alpha \partial_x )t} \tilde{u}(t)$ and carry out an approximation as above, i.e.,
\begin{equation*}
\tilde{v}(t_n+\tau) \approx \tilde{v}(t_n) + \frac{1}{2}\int_0^\tau \mathrm{e}^{(t_n+s) (\partial_x^3- \alpha \partial_x )} \partial_x \left (\mathrm{e}^{- (t_n+s) (\partial_x^3- \alpha \partial_x )} \tilde{v}(t_n)\right)^2 \mathrm{d}s.
\end{equation*}
The relation
\begin{equation*}
- (k_1+k_2)^3 + \alpha (k_1+k_2) +   k_1^3+ k_2^3 - \alpha k_1 -\alpha k_2 = -(k_1+k_2)^3 + k_1^3 + k_2^3
\end{equation*}
then allows us to derive similarly to above an exponential-type integration scheme
\begin{equation}\label{vtS}
\begin{aligned}
&\tilde v^{n+1}= \tilde v^n + \frac{1}{6}\mathrm{e}^{(\partial_{x}^3-\alpha \partial_x) (t_n+\tau)} \left( \mathrm{e}^{-(\partial_{x}^3-\alpha \partial_x) (t_n+\tau)} \partial_x^{-1}  \tilde{v}^n\right)^2 - \frac{1}{6}\mathrm{e}^{(\partial_{x}^3-\alpha \partial_x)  t_n} \left( \mathrm{e}^{-(\partial_{x}^3-\alpha \partial_x) t_n} \partial_x^{-1}\tilde v^n\right)^2,
\end{aligned}
\end{equation}
where $\partial_x^{-1}$ is defined in \eqref{nota} and $ \hat{ \tilde{v}}^{n+1}_0 = 0$ cf. Remark \ref{R:ZM-scheme}. Finally, by setting $u^{n}= \mathrm{e}^{- (\partial_x^3-\alpha \partial_x)t_n} \tilde{v}^{n}+\alpha$ we then obtain an approximation to the exact solution $u(t_n)$ of the KdV equation \eqref{eq:kdv} (with non-zero zero-mode) at time $t_n = n \tau$.
\end{remark}

Note that higher-order approximations to the solution of the KdV equation \eqref{eq:kdv} can be obtained by truncating the expansion in \eqref{To}  later. In Remark \ref{rem:2ord} below we explain the construction of a second-order scheme in more detail.
\begin{remark}[A second-order exponential-type integrator]\label{rem:2ord}
In order to derive a second-order approximation in the ``twisted" variable $v (t)$ we need to include the second-order term in the Taylor-series expansion of $v(t_n+s)$ in \eqref{To}. More precisely, plugging the formal expansion
\begin{equation*}\label{2oa}
v(t_n+s) = v(t_n) + s v'(t_n) + \mathcal{O}\left(s^2 v''\right)
\end{equation*}
into Duhamel's formula \eqref{To} yields  that
\begin{equation}\label{2so}
\begin{aligned}
 v(t_n+\tau)  & =  v(t_n) + \frac{1}{2}\int_0^\tau \mathrm{e}^{(t_n+s) \partial_{x}^3} \partial_x\left (\mathrm{e}^{-\partial_{x}^3 (t_n+s)} \big(v(t_n) + s v'(t_n)\big)\right)^2 \mathrm{d}s + \mathcal{R}_1(\tau,t_n,v) \\
& =   v(t_n) + \frac{1}{2}\int_0^\tau \mathrm{e}^{(t_n+s) \partial_{x}^3} \partial_x\Big[\left (\mathrm{e}^{-\partial_{x}^3 (t_n+s)} v(t_n) \right)^2\\&\qquad \qquad \qquad \qquad  \qquad \qquad + 2s  \left(\mathrm{e}^{-\partial_{x}^3 (t_n+s)} v(t_n) \right)  \left (\mathrm{e}^{-\partial_{x}^3 (t_n+s)} v'(t_n) \right) \Big]
 \mathrm{d}s
\\& + \mathcal{R}_1(\tau,t_n,v) + \mathcal{R}_2(\tau,t_n,v),
\end{aligned}
\end{equation}
where the remainders satisfy for $r>1/2$ and some constant $c>0$ that
\begin{equation}\label{2reg}
\begin{aligned}
\Vert \mathcal{R}_1(\tau,t_n,v)\Vert_r+\Vert \mathcal{R}_2(\tau,t_n,v)\Vert_r \leq c \tau^3 \sup_{t_n \leq t \leq t_{n+1}} \Big(\Vert \partial_x \big( v  v''\big)(t)\Vert_r + \Vert \partial_x (v')^2(t)\Vert_r \Big).
 \end{aligned}
\end{equation}
In order to construct a numerical scheme based on the expansion \eqref{2so} we need to solve the appearing integral. The first term involving the product $v^2$ can be easily determined thanks to~\eqref{int}. Note that $\hat{v}_0(t) = \hat{v}'_0(t) = 0$. Thus, similarly, we obtain for the  $v v'-$ term with the aid of the key-relation \eqref{eq:key} and integration by parts that
\begin{equation}\label{2cal}
\begin{aligned}
& \int_0^\tau s \cdot \mathrm{e}^{(t_n+s) \partial_{x}^3} \partial_x    \left(\mathrm{e}^{-\partial_{x}^3 (t_n+s)} v(t_n) \right)  \left (\mathrm{e}^{-\partial_{x}^3 (t_n+s)} v'(t_n) \right) \mathrm{d}s\\
& =  \sum_{k_1,k_2} \int_0^\tau s \cdot  \mathrm{e}^{-i(t_n+s) \big((k_1+k_2)^3-k_1^3-k_2^3\big)} i (k_1+k_2) \hat v_{k_1}(t_n) \hat v^\prime_{k_2}(t_n) \mathrm{e}^{i (k_1+k_2) x}\mathrm{d}s\\
& =  \sum_{\substack{k_1\neq 0 ,k_2 \neq 0\\k_1+k_2 \neq 0}} \int_0^\tau s \cdot  \mathrm{e}^{-i(t_n+s) \big((k_1+k_2)^3-k_1^3-k_2^3\big)} i (k_1+k_2) \hat v_{k_1}(t_n) \hat v^\prime_{k_2}(t_n) \mathrm{e}^{i (k_1+k_2) x}\mathrm{d}s\\
& =\tau  \sum_{\substack{k_1\neq 0 ,k_2 \neq 0\\k_1+k_2 \neq 0}}  \mathrm{e}^{-i(t_n+\tau) \big((k_1+k_2)^3-k_1^3-k_2^3\big)}\frac{1}{-3k_1k_2} \hat v_{k_1}(t_n) \hat v^\prime_{k_2}(t_n) \mathrm{e}^{i (k_1+k_2) x} \\
&- \sum_{\substack{k_1\neq 0 ,k_2 \neq 0\\k_1+k_2 \neq 0}} \int_0^\tau   \mathrm{e}^{-i(t_n+s) \big((k_1+k_2)^3-k_1^3-k_2^3\big)} \frac{1}{-3k_1k_2} \hat v_{k_1}(t_n) \hat v^\prime_{k_2}(t_n) \mathrm{e}^{i (k_1+k_2) x}\mathrm{d}s\\
& =\tau \sum_{\substack{k_1\neq 0 ,k_2 \neq 0\\k_1+k_2 \neq 0}}   \mathrm{e}^{-i(t_n+\tau) \big((k_1+k_2)^3-k_1^3-k_2^3\big)}\frac{1}{-3k_1k_2} \hat v_{k_1}(t_n) \hat v^\prime_{k_2}(t_n) \mathrm{e}^{i (k_1+k_2) x} \\
& -  \sum_{\substack{k_1\neq 0 ,k_2 \neq 0\\k_1+k_2 \neq 0}}\frac{\mathrm{e}^{-i(t_n+\tau) \big((k_1+k_2)^3-k_1^3-k_2^3\big)}- \mathrm{e}^{-i t_n \big((k_1+k_2)^3-k_1^3-k_2^3\big)}}{9 i k_1^2k_2^2 (k_1+k_2)} \hat v_{k_1}(t_n) \hat v^\prime_{k_2}(t_n) \mathrm{e}^{i (k_1+k_2) x}\mathrm{d}s\\
& = \frac{\tau}{3} \mathrm{e}^{(t_n + \tau)\partial_x^3}    \left(\mathrm{e}^{-\partial_{x}^3 (t_n+\tau)}\partial_x^{-1} v(t_n) \right)  \left (\mathrm{e}^{-\partial_{x}^3 (t_n+\tau)} \partial_x^{-1}v'(t_n) \right)\\
& \quad -\frac{1}{9}\mathrm{e}^{(t_n+\tau)\partial_x^3} \partial_x^{-1} \left( \mathrm{e}^{-(t_n+\tau)\partial_x^3 }(\partial_x^{-1})^2 v(t_n)\right) \left( \mathrm{e}^{-(t_n+\tau) \partial_x^3} (\partial_x^{-1})^2 v'(t_n)\right) \\&\quad + \frac{1}{9}\mathrm{e}^{t_n\partial_x^3} \partial_x^{-1} \left( \mathrm{e}^{-t_n\partial_x^3 }(\partial_x^{-1})^2 v(t_n)\right) \left( \mathrm{e}^{-t_n \partial_x^3} (\partial_x^{-1})^2 v'(t_n)\right)
\end{aligned}
\end{equation}
with $\partial_x^{-1}$ defined in \eqref{nota}. Plugging the relations given in \eqref{int} and \eqref{2cal} together with the definition
\begin{equation}\label{vprime}
{v^\prime}^n := \frac12 \mathrm{e}^{t_n \partial_x^3}\partial_x \left( \mathrm{e}^{-t_n \partial_x^3} v^n\right)^2
\end{equation}
 (see \eqref{kdvT}) into the expansion \eqref{2so} builds the basis of our numerical scheme: As a second-order approximation to the solution $v(t_n+\tau)$ of \eqref{kdvT} we take the  exponential-type integration scheme
\begin{equation*}\label{2scheme}
\begin{aligned}
v^{n+1}&= v^n + \frac{1}{6}\mathrm{e}^{\partial_{x}^3 (t_n+\tau)} \left( \mathrm{e}^{-\partial_{x}^3(t_n+\tau)} \partial_x^{-1}v^n\right)^2 - \frac{1}{6}\mathrm{e}^{\partial_{x}^3 t_n} \left( \mathrm{e}^{-\partial_{x}^3t_n} \partial_x^{-1}v^n\right)^2\\
& +  \frac{\tau}{3} \mathrm{e}^{(t_n + \tau)\partial_x^3}    \left(\mathrm{e}^{-\partial_{x}^3 (t_n+\tau)}\partial_x^{-1} v^n\right)  \left (\mathrm{e}^{-\partial_{x}^3 (t_n+\tau)} \partial_x^{-1}{v^\prime}^n \right)\\
& - \frac{1}{9}\mathrm{e}^{(t_n+\tau)\partial_x^3} \partial_x^{-1} \left( \mathrm{e}^{-(t_n+\tau)\partial_x^3 }(\partial_x^{-1})^2 v^n\right) \left( \mathrm{e}^{-(t_n+\tau) \partial_x^3} (\partial_x^{-1})^2 {v^\prime}^n\right)  \\& + \frac{1}{9}\mathrm{e}^{t_n\partial_x^3} \partial_x^{-1} \left( \mathrm{e}^{-t_n\partial_x^3 }(\partial_x^{-1})^2 v^n\right) \left( \mathrm{e}^{-t_n \partial_x^3} (\partial_x^{-1})^2 {v^\prime}^n \right)
\end{aligned}
\end{equation*}
with ${v^\prime}^n$ given in \eqref{vprime}, $\partial_x^{-1}$ defined in \eqref{nota} and by construction $\hat{v}_0^{n+1} = 0$ (cf. Remark \ref{R:ZM-scheme}). 

The approximation to the original solution $u(t_n)$ of the KdV equation~\eqref{eq:kdv} at time $t_n = n \tau$ is then obtained by ``twisting" the variable back again, i.e., by setting $u^{n}= \mathrm{e}^{- \partial_x^3 t_n} v^{n}$. This yields that
\begin{equation}\label{2scheme}
\begin{aligned}
u^{n+1}&= \mathrm{e}^{-\partial_x^3 \tau} u^n + \frac{1}{6}\left( \mathrm{e}^{-\partial_{x}^3 \tau} \partial_x^{-1}u^n\right)^2 - \frac{1}{6}\mathrm{e}^{-\partial_{x}^3 \tau} \left( \partial_x^{-1}u^n\right)^2 +  \frac{\tau}{3}    \left(\mathrm{e}^{-\partial_{x}^3\tau}\partial_x^{-1} u^n\right)  \left (\mathrm{e}^{-\partial_{x}^3\tau }{\textstyle \frac12} (u^n)^2 \right)\\
& - \frac{1}{9}\partial_x^{-1} \left( \mathrm{e}^{-\partial_x^3\tau }(\partial_x^{-1})^2 u^n\right) \left( \mathrm{e}^{-\tau\partial_x^3} \partial_x^{-1} {\textstyle \frac12} (u^n)^2\right)   + \frac{1}{9}\partial_x^{-1}  \mathrm{e}^{-\partial_x^3\tau } \partial_x^{-1} \big( (\partial_x^{-1})^2 u^n\big) \left( \partial_x^{-1} {\textstyle \frac12} (u^n)^2 \right),
\end{aligned}
\end{equation}
where $\partial_x^{-1}$ is defined in \eqref{nota} and by construction $\hat{u}_0^{n+1} = 0$ (cf. Remark \ref{R:ZM-scheme}). The  semi-discret exponential-type integration scheme \eqref{2scheme} \emph{allows formally second-order convergence} in $H^r$ for sufficiently smooth solutions $u(t) \in H^{r+5}$ with $r>1/2$ thanks to the local error bound \eqref{2reg} together with the observation that
$$
\Vert v''(t) \Vert_r \leq c \left(\Vert \partial_x^2 v (t)\Vert_r^3 + \Vert \partial_x^4 v (t)\Vert_r^2\right) = c  \left(\Vert \partial_x^2 u (t)\Vert_r^3 + \Vert \partial_x^4 u (t)\Vert_r^2\right).
$$
We do not pursuit this here and only underline the second-order convergence rate numerically in Section \ref{sec:num}.
\end{remark}
In the following section we give a detailed convergence analysis of the first-order exponential-type integration scheme \eqref{schemeU}.

\section{Error analysis}
For simplicity we carry out the error analysis for initial values satisfying Assumption \ref{as0}. Furthermore, in the following we denote by $\langle \cdot, \cdot \rangle$ the $L^2$ scalar product, i.e., $\langle f,g\rangle = \int_{\mathbb{T}} f g \mathrm{d}x$ and by $\Vert \cdot \Vert_{L^2}$ the corresponding $L^2$ norm. 

In order to  obtain a convergence result in $H^1$ we follow the strategy presented in \cite{Lubich08,HLR12}: We first prove convergence order of one half of the numerical scheme \eqref{scheme} in $H^2$ for solutions in $H^3$, see Section \ref{sec:h2} Theorem \ref{thm:main0}. This yields essential a priori bounds on the numerical solution in $H^2$ and allows us to prove first-order convergence globally in $H^1$, see Theorem \ref{thm:main} in Section~\ref{sec:gbH1}. The latter in particular implies first-order convergence of the exponential-type integration scheme \eqref{schemeU} towards the KdV solution~\eqref{eq:kdv}, see Corollary \ref{cor:ucon} below for the precise convergence result.

\subsection{Error analysis in $H^2$}\label{sec:h2}
We commence with the error analysis of the numerical scheme~\eqref{scheme} in $H^2$. In Section \ref{sec:stab} we carry out the stability analysis in $H^2$. In Section \ref{sec:lerr} we show that the method is consistent of order one half in $H^2$ for solutions in $H^3$.
\subsubsection{Stability analysis}\label{sec:stab}
 Set 
\begin{align}\label{phiS}
\Phi^\tau_t(v) :=v + \frac{1}{6}\mathrm{e}^{\partial_{x}^3 (t+\tau)} \left( \mathrm{e}^{-\partial_{x}^3(t+\tau)} \partial_x^{-1}v\right)^2 - \frac{1}{6}\mathrm{e}^{\partial_{x}^3 t} \left( \mathrm{e}^{-\partial_{x}^3t} \partial_x^{-1}v\right)^2
\end{align}
such that for all $k$ we have $v^{k+1}= \Phi^\tau_{t_k} (v^k)$. The following stability result holds for the numerical flow $\Phi^\tau_t$:
\begin{lemma}\label{lem:stab}
Let $f \in H^2$ and $g \in H^3$. Then, for all $t \in \mathbb{R}$ we have
\begin{equation*}
\left \Vert \partial_x^2( \Phi^\tau_t(f) - \Phi^\tau_t(g)) \right \Vert_{L^2}  \leq \mathrm{exp}( \tau L) \Vert \partial_x^2( f-g)\Vert_{L^2},
\end{equation*}
where $L$ depends on $\Vert \partial_x^2 f \Vert_{L^2}$ and $\Vert \partial_x^3 g \Vert_{L^2}$.
\end{lemma}
\begin{proof}
Note that
\begin{equation*}
\begin{aligned}\label{eq:Lf}
& \Vert \partial_x^2(\Phi^\tau_t(f) - \Phi^\tau_t(g)) \Vert_{L^2}^2  = \Vert \partial_x^2 (f - g )\Vert_{L^2}^2\\
& + \frac{1}{3} \langle \partial_x^2  \mathrm{e}^{\partial_{x}^3 (t+\tau)}\Big [  \left( \mathrm{e}^{-\partial_{x}^3(t+\tau)} \partial_x^{-1}f\right)^2- \left( \mathrm{e}^{-\partial_{x}^3(t+\tau)} \partial_x^{-1}g\right)^2\Big], \partial_x^2 (f-g)\rangle\\&
 - \frac{1}{3}\langle \partial_x^2 \mathrm{e}^{\partial_{x}^3 t} \Big[ \left( \mathrm{e}^{-\partial_{x}^3t} \partial_x^{-1}f\right)^2- \left( \mathrm{e}^{-\partial_{x}^3t} \partial_x^{-1}g\right)^2\Big], \partial_x^2 (f-g)\rangle\\
 & + \frac{1}{6^2} \Vert 
 \partial_x^2  \mathrm{e}^{\partial_{x}^3 (t+\tau)}\Big [  \left( \mathrm{e}^{-\partial_{x}^3(t+\tau)} \partial_x^{-1}f\right)^2- \left( \mathrm{e}^{-\partial_{x}^3(t+\tau)} \partial_x^{-1}g\right)^2\Big]\\&
 -
 \partial_x^2 \mathrm{e}^{\partial_{x}^3 t} \Big[ \left( \mathrm{e}^{-\partial_{x}^3t} \partial_x^{-1}f\right)^2- \left( \mathrm{e}^{-\partial_{x}^3t} \partial_x^{-1}g\right)^2\Big]\Vert_{L^2}^2\\
 & =: \Vert \partial_x^2 (f-g)\Vert_{L^2}^2 + \frac{1}{3}I_1 + \frac{1}{6^2} I_2.
\end{aligned}
\end{equation*}
Lemma \ref{lem:I1} and Lemma \ref{lem:I2} below allow us the following bounds on $I_1$ and $I_2$: We have
\begin{equation}\label{I12}
\vert I_1+I_2 \vert \leq \tau L \Vert \partial_x^2 (f-g)\Vert_{L^2}^2,
\end{equation}
where $L$ depends on $\Vert \partial_x^2 f \Vert_{L^2}$ and $\Vert \partial_x^3 g \Vert_{L^2}$. Hence,
\begin{equation*}
\Vert \partial_x^2\big(\Phi^\tau_t(f) - \Phi^\tau_t(g)\big) \Vert_{L^2}^2 \leq (1 + \tau L ) \Vert \partial_x^2 (f-g)\Vert_{L^2}^2
\end{equation*}
which yields the assertion.
\end{proof}
In the rest of Section \ref{sec:stab} we will show the essential bound \eqref{I12}. We start with a useful Lemma.
\begin{lemma}\label{lem:kb}
The following estimates hold for $u,v,w \in H^2$
\begin{equation}\label{bb}\begin{aligned}
&\vert \langle \partial_x^2 u, v w - \mathrm{e}^{\partial_x^3 \tau} \Big[ \left(\mathrm{e}^{-\partial_x^3 \tau} v \right)  \left(\mathrm{e}^{-\partial_x^3 \tau}w \right)\Big] \rangle \vert  \leq c \tau \Vert \partial_x^2 u \Vert_{L^2} \Vert \partial_x^2 v \Vert_{L^2}\Vert \partial_x^2 w \Vert_{L^2}\\
&\vert \langle  u, (\partial_x v)^2 - \mathrm{e}^{\partial_x^3 \tau} \left(\mathrm{e}^{-\partial_x^3 \tau} \partial_x v \right)^2 \rangle \vert  \leq c \tau \Vert \partial_x^2 u \Vert_{L^2} \Vert \partial_x^2 v \Vert_{L^2}^2
\end{aligned}
\end{equation}
\end{lemma}
for some constant $c> 0$.
\begin{proof}
The key relation \eqref{eq:key} together with the Cauchy-Schwarz inequality allows us the following bound
\begin{equation}\label{a1}
\begin{aligned}
& \vert \langle \partial_x^2 u  , v w - \mathrm{e}^{\partial_x^3 \tau} \Big[ \left(\mathrm{e}^{-\partial_x^3 \tau} v \right)  \left(\mathrm{e}^{-\partial_x^3 \tau}w \right)\Big] \rangle \vert  \\
&= \vert \sum_{k_1,k_2} (k_1+k_2)^2  \hat{u}_{-(k_1+k_2)} \left( 1 - \mathrm{e}^{- i\tau \big( (k_1+k_2)^3 - k_1^3 - k_2^3\big)}\right) \hat v_{k_1}\hat{w}_{k_2} \vert \\
& = \vert \sum_{k_1,k_2}(k_1+k_2)^2 \hat{u}_{-(k_1+k_2)} \left( 1 - \mathrm{e}^{- i \tau 3 k_1 k_2 (k_1+k_2)}\right) \hat v_{k_1}\hat{w}_{k_2} \vert \\
& \leq 3 \tau  \sum_{k_1,k_2} \vert (k_1+k_2)^2  \hat{u}_{-(k_1+k_2)}\vert \vert (k_1+k_2) k_1k_2\hat v_{k_1}  \hat{w}_{k_2}\vert \\
&  = 3 \tau \sum_{l,k}l^2 \vert \hat{u}_{-l} \vert \vert l  k(l-k) \vert\vert \hat{v}_{k} \hat{w}_{l-k}\vert\\
& \leq 3 \tau \sum_{l,k}l^2 \vert \hat{u}_{-l} \vert  \left( \vert k ( l-k)^2\vert  \vert  \hat{v}_{k} \hat{w}_{l-k}\vert + \vert k\vert^2\vert l - k\vert \vert \hat{v}_{k} \hat{w}_{l-k}\vert \right)\\
& \leq 3 \tau \big ( \sum_l  l^4 \vert \hat{u}_{l}\vert^2\big)^{1/2} \Big( \sum_l \big( \sum_k \vert k\vert \vert\hat{v}_k\vert \vert l-k\vert^2 \vert \hat{w}_{l-k}\vert\big)^2 \Big)^{1/2}\\
& + 3 \tau \big ( \sum_l  l^4 \vert \hat{u}_{l}\vert^2\big)^{1/2}  \Big( \sum_l \big(\sum_k \vert k\vert^2 \vert\hat{v}_k\vert \vert l - k\vert \vert \hat{w}_{l-k}\vert\big)^2 \Big)^{1/2}\\
& \leq 3 \tau \Vert \partial_x^2 u \Vert_{L^2} \big( \Vert v^{(1)} \ast w^{(2)}\Vert_{l^2} + \Vert v^{(2)} \ast w^{(1)} \Vert_{l^2}\big),
\end{aligned}
\end{equation}
where $v^{(j)}(k) := |k|^j \vert \hat{v}_k\vert $ and $w^{(j)}(k) : = |k|^j \vert \hat{w}_k\vert$. By the Young and Cauchy-Schwarz inequality we furthermore obtain that
\begin{equation}\label{a2}
\begin{aligned}
\Vert v^{(1)} \ast w^{(2)}\Vert_{l^2} + \Vert v^{(2)} \ast w^{(1)} \Vert_{l^2} &\leq \Vert v^{(1)}\Vert_{l^1} \Vert w^{(2)}\Vert_{l^2} + \Vert w^{(1)}\Vert_{l^1} \Vert v^{(2)}\Vert_{l^2} \leq c \Vert v^{(2)} \Vert_{l^2} \Vert w^{(2)}\Vert_{l^2} \\
& \leq c \Vert \partial_x^2 v\Vert_{L^2}\Vert \partial_x^2 w \Vert_{L^2}
\end{aligned}
\end{equation}
for some constant $c>0$. Plugging \eqref{a2} into \eqref{a1} yields the first assertion.

Similarly we have that
\begin{equation}
\begin{aligned}
& \vert \langle  u, (\partial_x v)^2 - \mathrm{e}^{\partial_x^3 \tau} \left(\mathrm{e}^{-\partial_x^3 \tau} \partial_x v \right)^2 \rangle \vert \leq 3 \tau \sum_{k_1,k_2} \vert (k_1+k_2) \hat{u}_{-(k_1+k_2)} \vert \vert k_1 k_2\vert^2  \vert \hat{v}_{k_1}\hat{v}_{k_2}\vert \\
& = 3 \tau \sum_{k,l}\vert l \vert \vert \hat{u}_{-l}\vert k^2 (l-k)^2 \vert \hat{v}_k \hat{v}_{k-l}\vert\leq 3 \tau \Big( \sum_{k} k^4 \vert \hat{v}_k\vert^2\Big)^{1/2} \Vert u^{(1)} \ast v^{(2)}\Vert_{l^2} \\
& \leq c \tau \Vert \partial_x^2 v \Vert_{L^2} \Vert u^{(1)}\Vert_{l^1} \Vert v^{(2)}\Vert_{l^2} \leq c \tau \Vert \partial_x^2 v \Vert_{L^2}^2 \Vert \partial_x^2 u \Vert_{L^2}
\end{aligned}
\end{equation}
which yields the second assertion.
\end{proof}
\begin{lemma}[Bound on $I_1$]\label{lem:I1}
We have
\begin{equation*}
\vert I_1 \vert \leq c \tau \Big(\Vert \partial_x^2 (f-g) \Vert_{L^2} + \Vert \partial_x^3 g\Vert_{L^2} \Big) \Vert \partial_x^2 (f - g) \Vert_{L^2}^2 
\end{equation*}
for some constant $c > 0$.
\end{lemma}
\begin{proof}
Note that for all $ t \in \mathbb{R}$ the following relation holds
\begin{equation*}
\langle \mathrm{e}^{t \partial_x^3} f, g \rangle = \langle  f, \mathrm{e}^{-t \partial_x^3} g \rangle .
\end{equation*}
Thus, by setting $\tilde{f}= \mathrm{e}^{-t\partial_x^3} f$ and $\tilde{g} = \mathrm{e}^{-t\partial_x^3} g$ we obtain that
\begin{equation*}
\begin{aligned}
 I_1 
 & =
 \langle \partial_x^2 \left( \mathrm{e}^{-\partial_{x}^3\tau } \partial_x^{-1}\tilde{f}\right)^2- \partial_x ^2\left( \mathrm{e}^{-\partial_{x}^3\tau } \partial_x^{-1}\tilde{g}\right)^2,  \mathrm{e}^{-\partial_{x}^3\tau } \partial_x^2 (\tilde f-\tilde g)\rangle\\&
 - \langle \partial_x^2 \left(\partial_x^{-1} \tilde{f}\right)^2- \partial_x^2 \left( \partial_x^{-1}\tilde{g}\right)^2, \partial_x^2 (\tilde{f}-\tilde{g})\rangle.
 \end{aligned}
 \end{equation*}
 Using the relation $f^2- g^2 = (f-g)^2 + 2(f-g) g$ as well as the chain rule yields that
 \begin{equation*}
 \begin{aligned}
 I_1  & =
 \langle \partial_x^2 \left( \mathrm{e}^{-\partial_{x}^3\tau } \partial_x^{-1}(\tilde{f}-\tilde{g})\right)^2+ 2 \partial_x^2\left( \mathrm{e}^{-\partial_{x}^3\tau } \partial_x^{-1}(\tilde{f}-\tilde{g})\right)\left( \mathrm{e}^{-\partial_{x}^3\tau } \partial_x^{-1}\tilde{g}\right) ,  \mathrm{e}^{-\partial_{x}^3\tau } \partial_x^2 (\tilde f-\tilde g)\rangle\\&
 - \langle \partial_x^2 \left(\partial_x^{-1} (\tilde{f}-\tilde{g}) \right)^2+ 2 \partial_x ^2 \left(\partial_x^{-1} (\tilde{f}-\tilde{g}) \right)\left( \partial_x^{-1}\tilde{g}\right), \partial_x ^2(\tilde{f}-\tilde{g})\rangle\\
    & =
2  \langle \left( \mathrm{e}^{-\partial_{x}^3\tau } (\tilde{f}-\tilde{g})\right)^2+\left( \mathrm{e}^{-\partial_{x}^3\tau } \partial_x^{-1}(\tilde{f}-\tilde{g})\right) \left( \mathrm{e}^{-\partial_{x}^3\tau } \partial_x (\tilde{f}-\tilde{g})\right),  \mathrm{e}^{-\partial_{x}^3\tau } \partial_x^2 (\tilde f-\tilde g)\rangle
\\
&+  2 \langle\left( \mathrm{e}^{-\partial_{x}^3\tau }\partial_x (\tilde{f}-\tilde{g})\right)\left( \mathrm{e}^{-\partial_{x}^3\tau } \partial_x^{-1}\tilde{g}\right)+ 2 \left( \mathrm{e}^{-\partial_{x}^3\tau }(\tilde{f}-\tilde{g})\right)\left( \mathrm{e}^{-\partial_{x}^3\tau } \tilde{g}\right) ,  \mathrm{e}^{-\partial_{x}^3\tau } \partial_x^2 (\tilde f-\tilde g)\rangle
\\
&+  2 \langle \left( \mathrm{e}^{-\partial_{x}^3\tau }\partial_x^{-1}(\tilde{f}-\tilde{g})\right)\left( \mathrm{e}^{-\partial_{x}^3\tau } \partial_x \tilde{g}\right) ,  \mathrm{e}^{-\partial_{x}^3\tau } \partial_x^2 (\tilde f-\tilde g)\rangle
\\
&
 - 2 \langle  \left( \partial_x (\tilde{f}-\tilde{g}) \right)  \left(\partial_x^{-1}(\tilde{f}-\tilde{g}) \right) +( \tilde{f}-\tilde{g}) ^2, \partial_x^2 (\tilde{f}-\tilde{g})\rangle
 \\ &  
 - 2 \langle   \left(\partial_x^{-1} (\tilde{f}-\tilde{g}) \right)\partial_x\tilde{g} + 2( \tilde{f}-\tilde{g} )\tilde{g}+ \left(\partial_x (\tilde{f}-\tilde{g}) \right) \left(\partial_x^{-1}\tilde{g}\right)  , \partial_x^2 (\tilde{f}-\tilde{g})\rangle  .
  \end{aligned}
 \end{equation*}
Next we use another key fact namely that
 $$
 \langle v u, \partial_x u\rangle = \frac{1}{2} \langle v, \partial_x (u)^2\rangle = - \frac{1}{2}\langle \partial_x v, u^2\rangle
 $$
as well as that
 $$
 \langle u - v, \left(\partial_x (u-v)\right)^2\rangle = - \frac{1}{2}\langle (u-v)^2, \partial_x^2 (u-v)\rangle.
 $$
 This yields that
 \begin{equation*}
 \begin{aligned}
 I_1   & =
2  \langle \left( \mathrm{e}^{-\partial_{x}^3\tau } (\tilde{f}-\tilde{g})\right)^2,  \mathrm{e}^{-\partial_{x}^3\tau } \partial_x^2 (\tilde f-\tilde g)\rangle
\\
& +\frac{1}{2} \langle \left( \mathrm{e}^{-\partial_{x}^3\tau }(\tilde{f}-\tilde{g})\right)^2 , \partial_x^2  \mathrm{e}^{-\partial_{x}^3\tau } (\tilde f-\tilde g)\rangle
\\
&- \langle \mathrm{e}^{-\partial_{x}^3\tau }\tilde{g},\left(  \mathrm{e}^{-\partial_{x}^3\tau } \partial_x (\tilde f-\tilde g)\right)^2\rangle
\\
&+4 \langle \left( \mathrm{e}^{-\partial_{x}^3\tau }(\tilde{f}-\tilde{g})\right)\left( \mathrm{e}^{-\partial_{x}^3\tau } \tilde{g}\right) ,  \mathrm{e}^{-\partial_{x}^3\tau } \partial_x^2 (\tilde f-\tilde g)\rangle
\\
&+  2 \langle \left( \mathrm{e}^{-\partial_{x}^3\tau }\partial_x^{-1}(\tilde{f}-\tilde{g})\right)\left( \mathrm{e}^{-\partial_{x}^3\tau } \partial_x \tilde{g}\right) ,  \mathrm{e}^{-\partial_{x}^3\tau } \partial_x^2 (\tilde f-\tilde g)\rangle
\\
&
 -\frac{1}{2}\langle  (\tilde{f}-\tilde{g})^2,\partial_x^2 (\tilde{f}-\tilde{g})\rangle
 \\ &  
  - 2 \langle ( \tilde{f}-\tilde{g})^2, \partial_x^2 (\tilde{f}-\tilde{g})\rangle
 \\ &  
 - 2 \langle   \left(\partial_x^{-1} (\tilde{f}-\tilde{g}) \right)\partial_x\tilde{g}  , \partial_x^2 (\tilde{f}-\tilde{g})\rangle 
  \\ &  
 - 4 \langle ( \tilde{f}-\tilde{g} )\tilde{g}, \partial_x^2 (\tilde{f}-\tilde{g})\rangle  
  \\ &  
 + \langle  \tilde{g}  , \left( \partial_x (\tilde{f}-\tilde{g})\right)^2\rangle  . 
 \end{aligned}
 \end{equation*}
 Thus, rearranging the terms leads to
   \begin{equation*}
   \begin{aligned}
 I_1 & = \frac{5}{2} \langle   \mathrm{e}^{\partial_{x}^3\tau }  \left(  \mathrm{e}^{-\partial_{x}^3\tau } (\tilde f-\tilde g)\right)^2 -(\tilde{f}-\tilde{g})^2,\partial_x^2( \tilde{f}-\tilde{g})\rangle
\\
&-  \langle  \mathrm{e}^{\partial_{x}^3\tau } \left(\mathrm{e}^{-\partial_{x}^3\tau } \partial_x (\tilde f-\tilde g)\right)^2 - \left(\partial_x (\tilde{f}-\tilde{g})\right)^2,\tilde{g}\rangle 
\\&
+ 4\langle  \mathrm{e}^{\partial_{x}^3\tau }  \left(  \mathrm{e}^{-\partial_{x}^3\tau } (\tilde{f}-\tilde{g})\right)\left( \mathrm{e}^{-\partial_{x}^3\tau } \tilde{g}\right) - (\tilde{f}-\tilde{g}) \left( \tilde{g}\right)  , \partial_x^2 (\tilde f-\tilde g)\rangle
 \\ &   
+2\langle  \mathrm{e}^{\partial_{x}^3\tau }  \left(  \mathrm{e}^{-\partial_{x}^3\tau } \partial_x^{-1}(\tilde{f}-\tilde{g})\right)\left( \mathrm{e}^{-\partial_{x}^3\tau } \partial_x \tilde{g}\right) - \left(\partial_x^{-1}(\tilde{f}-\tilde{g})\right) \left( \partial_x\tilde{g}\right)  , \partial_x^2 (\tilde f-\tilde g)\rangle.
 \end{aligned}
\end{equation*}
With the aid of Lemma \ref{lem:kb} we thus obtain that
\begin{equation*}
\vert I_1 \vert \leq  \tau c \Big(\Vert \partial_x^2 (f-g) \Vert_{L^2} + \Vert \partial_x^3 g\Vert_{L^2} \Big) \Vert \partial_x^2 (f - g) \Vert_{L^2}^2 
\end{equation*}
for some constant $c > 0$.
\end{proof}
\begin{lemma}[Bound on $I_2$] \label{lem:I2} We have
\begin{equation*}
\vert I_2 \vert  \leq  \tau  M \Vert \partial_x^2 (f-g)\Vert_{L^2}^2,
\end{equation*}
where $M$ depends on $\Vert \partial_x^2 f \Vert_{L^2}$ and $\Vert \partial_x^2 g \Vert_{L^2}$.
\end{lemma}
\begin{proof}
In the following let $M$ denote a constant depending on $\Vert \partial_x^2 f \Vert_{L^2}$ and $\Vert \partial_x^2 g \Vert_{L^2}$. Setting $\tilde{f}= \mathrm{e}^{-t \partial_x^3} f$ and $\tilde{g}= \mathrm{e}^{-t \partial_x^3} g$ yields that
\begin{equation}\label{I2}
\begin{aligned}
& I_2 = \langle  
  \partial_x^2   \left( \mathrm{e}^{-\partial_{x}^3\tau } \partial_x^{-1}\tilde{f}\right)^2- \partial_x^2\left( \mathrm{e}^{-\partial_{x}^3\tau } \partial_x^{-1}\tilde{g}\right)^2,
   \partial_x^2   \left( \mathrm{e}^{-\partial_{x}^3 \tau } \partial_x^{-1}\tilde{f}\right)^2-\partial_x^2 \left( \mathrm{e}^{-\partial_{x}^3\tau} \partial_x^{-1}\tilde{g}\right)^2 \rangle
 \\&
 - 2 \langle \partial_x^2   \mathrm{e}^{\partial_{x}^3 \tau} \Big[ \left( \mathrm{e}^{-\partial_{x}^3 \tau} \partial_x^{-1} \tilde{f}\right)^2- \left( \mathrm{e}^{-\partial_{x}^3\tau} \partial_x^{-1}\tilde{g}\right)^2 \Big],
 \partial_x^2 \left(\partial_x^{-1}\tilde{f}\right)^2-\partial_x^2 \left(  \partial_x^{-1}\tilde{g}\right)^2
 \rangle
 \\
&
+ \langle \partial_x^2  \left( \partial_x^{-1} \tilde{f} \right)^2- \partial_x^2\left( \partial_x^{-1}\tilde{g}\right)^2
, \partial_x ^2\left(\partial_x^{-1}\tilde{f}\right)^2- \partial_x ^2 \left(\partial_x^{-1}\tilde{g}\right)^2
\rangle\\
& = I_2^a + I_2^b
 \end{aligned}
 \end{equation}
 with
 \begin{equation*}
 \begin{aligned}
 I_2^a  =& \langle  
  \partial_x^2   \left( \mathrm{e}^{-\partial_{x}^3\tau } \partial_x^{-1}\tilde{f}\right)^2- \partial_x^2\left( \mathrm{e}^{-\partial_{x}^3\tau } \partial_x^{-1}\tilde{g}\right)^2- \partial_x^2 \mathrm{e}^{-\partial_{x}^3 \tau} \Big[ \left(\partial_x^{-1}\tilde{f}\right)^2- \left(  \partial_x^{-1}\tilde{g}\right)^2\Big]
,\\&\quad
   \partial_x^2   \left( \mathrm{e}^{-\partial_{x}^3 \tau } \partial_x^{-1}\tilde{f}\right)^2-\partial_x^2 \left( \mathrm{e}^{-\partial_{x}^3\tau} \partial_x^{-1}\tilde{g}\right)^2 \rangle\\
  I_2^b  =&  -  \langle \partial_x^2   \mathrm{e}^{\partial_{x}^3 \tau} \Big[ \left( \mathrm{e}^{-\partial_{x}^3 \tau} \partial_x^{-1} \tilde{f}\right)^2- \left( \mathrm{e}^{-\partial_{x}^3\tau} \partial_x^{-1}\tilde{g}\right)^2 \Big] - \partial_x^2  \left( \partial_x^{-1} \tilde{f} \right)^2- \partial_x^2\left( \partial_x^{-1}\tilde{g}\right)^2,\\&
\quad \partial_x^2 \left(\partial_x^{-1}\tilde{f}\right)^2- \partial_x^2 \left(  \partial_x^{-1}\tilde{g}\right)^2
 \rangle.
 \end{aligned}
 \end{equation*}
 Similarly to Lemma \ref{lem:kb} we obtain with
 $F: = \partial_x^2 \left(\partial_x^{-1}\tilde{f}\right)^2- \partial_x^2  \left(\partial_x^{-1}\tilde{g}\right)^2$  by the key relation \eqref{eq:key} using the Cauchy-Schwarz and Young inequality that
 \begin{equation}\label{I2b}
 \begin{aligned}
  \vert I_2^b \vert & = \vert \sum_{k_1,k_2} \hat{F}_{-(k_1+k_2)} \frac{(k_1+k_2)^2 }{k_1 k_2} \left(1 - \mathrm{e}^{-i \tau \big((k_1+k_2)^3 - k_1^3-k_2^3\big)} \right)\left( \hat{\tilde{f}}_{k_1} \hat{\tilde{f}}_{k_2} - \hat{\tilde{g}}_{k_1} \hat{\tilde{g}}_{k_2}\right)\vert \\
 &  = \vert \sum_{k_1,k_2}\hat{F}_{-(k_1+k_2)} \frac{(k_1+k_2)^2 }{k_1 k_2} \left(1 - \mathrm{e}^{-i \tau 3 k_1k_2 (k_1+k_2)} \right)\left( \hat{\tilde{f}}_{k_1} \hat{\tilde{f}}_{k_2} - \hat{\tilde{g}}_{k_1} \hat{\tilde{g}}_{k_2}\right)\vert \\
 & \leq 3\tau   \sum_{k_1,k_2} \vert \hat{F}_{-(k_1+k_2)}(k_1+k_2)\vert  \vert (k_1+k_2)^2\left( \hat{\tilde{f}}_{k_1} \hat{\tilde{f}}_{k_2} - \hat{\tilde{g}}_{k_1} \hat{\tilde{g}}_{k_2}\right)\vert\\
 & \leq 3 \tau \sum_{k,l} \vert l \hat{F}_{-l}\vert \big( (l-k)^2 +2 |(l-k)k|  + k^2\big)\big \vert (\tilde{f}_{k} - \tilde{g}_k) \tilde{f}_{l-k} + \tilde{g}_k (\tilde{f}_{l-k} - \tilde{g}_{l-k})\big\vert\\
 & \leq 6 \tau \Vert \partial_x F\Vert_{L^2} \big(\sum_{j=0,1}\Vert (\tilde f-\tilde g)^{(2j)} \ast \tilde g^{(2-2j))}\Vert_{l^2} + \Vert (\tilde f-\tilde g)^{(2j)} \ast \tilde f^{(2-2j)}\Vert_{l^2}\big)\\
 & +  6 \tau \Vert \partial_x F\Vert_{L^2}\big( \Vert (\tilde f-\tilde g)^{(1)} \ast \tilde f^{(1)}\Vert_{l^2} + \Vert (\tilde f-\tilde g)^{(1)} \ast \tilde g^{(1)}\Vert_{l^2}\big)\\
 & \leq c \tau \Vert \partial_x F\Vert_{L^2} \Vert \partial_x^2 (f - g) \Vert_{L^2} (\Vert \partial_x^2 f \Vert_{L^2}+ \Vert \partial_x^2 g \Vert_{L^2}\big)\\
 & \leq  \tau M \Vert \partial_x^2 (f-g)\Vert_{L^2}^2,
  \end{aligned}
 \end{equation}
 where again we used the notation $\Phi^{(j)}(k) := |k|^j \vert \hat{\Phi}_k\vert$. Similarly, we obtain for $I_2^a$ with $F: =  \partial_x^2   \left( \mathrm{e}^{-\partial_{x}^3\tau } \partial_x^{-1}\tilde{f}\right)^2- \partial_x^2\left( \mathrm{e}^{-\partial_{x}^3\tau } \partial_x^{-1}\tilde{g}\right)^2$ that
 \begin{equation}\label{I2a}
 \begin{aligned}
 \vert I_2^a\vert 
 & = \vert \sum_{k_1,k_2}\hat{F}_{-(k_1+k_2)} \frac{(k_1+k_2)^2}{k_1k_2}\left( \mathrm{e}^{i \tau (k_1^3 + k_2^3)} - \mathrm{e}^{i \tau (k_1+k_2)^3}\right)\left( \hat{\tilde{f}}_{k_1} \hat{\tilde{f}}_{k_2} - \hat{\tilde{g}}_{k_1} \hat{\tilde{g}}_{k_2}\right)\vert\\
    & \leq \sum_{k_1,k_2} \vert\hat{F}_{-(k_1+k_2)} \frac{(k_1+k_2)^2}{k_1k_2} \vert \left \vert 1- \mathrm{e}^{-i \tau 3k_1k_2(k_1+k_2)}\right\vert \vert \left( \hat{\tilde{f}}_{k_1} \hat{\tilde{f}}_{k_2} - \hat{\tilde{g}}_{k_1} \hat{\tilde{g}}_{k_2}\right)\vert
     \\
 & \leq  \tau  M\Vert \partial_x^2 (f - g) \Vert_{L^2}^2.
 \end{aligned}
 \end{equation}
 Plugging the bounds \eqref{I2b} and \eqref{I2a} into \eqref{I2} yields the assertion.
\end{proof}
\subsubsection{Local error analysis}\label{sec:lerr}
Let $\phi^t$ denote the exact flow associated to the reformulated KdV equation~\eqref{kdvT}, i.e., $v(t) = \phi^t(v(0))$. The following local error bound holds for the exponential-type integrator $\Phi^\tau$ defined in \eqref{phiS} with $v^{k+1}= \Phi_{t_k}^\tau(v^k)$.
\begin{lemma}\label{lem:loc}
Let $v(t_k+t) = \phi^t(v(t_k)) \in H^3$ for $0 \leq t \leq \tau$. Then
\begin{equation*}
\Vert \partial_x^2 \big (\phi^\tau(v(t_k)) - \Phi^\tau_{t_k}(v(t_k) )\big)\Vert_{L^2} \leq c \tau^{3/2},
\end{equation*}
where $c$ depends on $\mathrm{sup}_{0\leq t \leq \tau}\Vert \phi^t(v(t_k)) \Vert_{H^3}$.
\end{lemma}
\begin{proof}
As $\mathrm{e}^{t\partial_x^3 }$ is a linear isometry in $H^r$ for all $t\in \mathbb{R}$ the iteration of Duhamel's formula \eqref{kdvT} yields that
\begin{equation}\label{i1}
\begin{aligned}
\Vert \phi^\tau(v(t_k)) - \Phi^\tau_{t_k}(v(t_k)) \Vert_{H^2} & \leq \int_0^\tau \Vert \left(\mathrm{e}^{-\partial_x^3 (t_k+s)} \phi^s(v(t_k))\right)^2 - \left(\mathrm{e}^{-\partial_x^3 (t_k+s)} v(t_k)\right)^2\Vert_{H^3}\mathrm{d}s\\
& \leq \tau c_1 \mathrm{sup}_{0 \leq t \leq \tau} \Vert \phi^t(v(t_k)) - v(t_k)\Vert_{H^3},
\end{aligned}
\end{equation}
where $c_1$ depends on $ \mathrm{sup}_{0 \leq t \leq \tau} \Vert \phi^t(v(t_k))\Vert_{H^3}$. Duhamel's formula \eqref{kdvT} and integration by parts furthermore yields that
\begin{equation}\label{i2}
\begin{aligned}
& \Vert \phi^t(v(t_j)) - v(t_j)\Vert_{H^3}  \leq \Vert \int_0^t \mathrm{e}^{(t_j+s)\partial_x^3} \partial_x \left( \mathrm{e}^{-\partial_x^3 (t_j+s)} v(t_j+s)\right)^2 \mathrm{d}s \Vert_{H^3}\\
& \leq \Vert \sum_{k_1,k_2} \frac{1}{k_1 k_2} \mathrm{e}^{- 3 i  t_j k_1k_2 (k_1+k_2)}\left( \mathrm{e}^{- 3 i  t k_1k_2 (k_1+k_2)}-1\right)  \hat{v}_{k_1}(t_j+t) \hat{v}_{k_2}(t_j+t)\mathrm{e}^{i (k_1+ k_2) x}\Vert_{H^3}\\
& + \Vert \sum_{k_1,k_2} \frac{1}{k_1 k_2}\mathrm{e}^{- 3 i  t_j k_1k_2 (k_1+k_2)} \big( \hat{v}_{k_1}(t_j+t) \hat{v}_{k_2}(t_j+t) -  \hat{v}_{k_1}(t_j) \hat{v}_{k_2}(t_j)\big) \mathrm{e}^{i (k_1+ k_2) x}\Vert_{H^3}\\
&+ \Vert \int_0^t \sum_{k_1,k_2} \mathrm{e}^{- 3 i  (t_j+s) k_1k_2 (k_1+k_2)} \frac{1}{k_1k_2} \frac{\mathrm{d}}{\mathrm{d}s}\big( \hat{v}_{k_1}(t_j+s) \hat{v}_{k_2}(t_k+s)\big) \mathrm{d}s \mathrm{e}^{i (k_1+ k_2) x}\Vert_{H^3}
\\
& \leq c t^{1/2} \sup_{k_1,k_2 \in \mathbb{Z}_{\neq 0}} \frac{\vert k_1+k_2\vert^{1/2}}{|k_1 k_2|^{1/2}} \Vert v(t_j+t)\Vert_{H^3}^2 +  c t  \sup_{k_1,k_2 \in \mathbb{Z}_{\neq 0}} \frac{|k_1\vert +\vert k_2|}{|k_1k_2|} \sup_{0 \leq s \leq t} \Vert v (t_j+s)\Vert_{H^3}^3.
\end{aligned}
\end{equation}
Plugging \eqref{i2} into \eqref{i1} yields the assertion.
\end{proof}
\subsubsection{Global error bound}\label{sec:gb}
The stability analysis in Section \ref{sec:stab} and local error analysis in Section \ref{sec:lerr} allows us  the following global error bound in $H^2$.
\begin{theorem}\label{thm:main0}
Let the solution of \eqref{kdvT} satisfy $v(t) \in H^3$ for $t \leq T$. Then there exists a $\tau_0 > 0$ such that for all $\tau \leq \tau_0$ and $t_n \leq T$ we have
\begin{equation*}
\Vert v(t_n) - v^n \Vert_{H^2} \leq c \tau^{1/2},
\end{equation*}
where $c$ depends on $\mathrm{sup}_{0 \leq t \leq t_n} \Vert v(t)\Vert_{H^3}$ and $t_n$, but can be chosen independently of $\tau$.
\end{theorem}
\begin{proof}
The triangular inequality yields that
\begin{equation}\label{inq}
\begin{aligned}
\Vert v(t_{k+1}) - v^{k+1}\Vert_{H^2} & = \Vert \phi^\tau(v(t_k) )- \Phi^\tau_{t_k}(v^k)\Vert_{H^2} \\
& \leq \Vert  \phi^\tau(v(t_k)) - \Phi^\tau_{t_k}(v(t_k)) \Vert_{H^2} + \Vert \Phi^\tau_{t_k}(v(t_k)) - \Phi^\tau_{t_k}(v^k)  \Vert_{H^2}.
\end{aligned}
\end{equation}
Thus, iterating the estimate \eqref{inq} we obtain with the aid of Lemma \ref{lem:stab} (with $g = v(t_k) \in H^3$) and Lemma \ref{lem:loc} that as long as $ v^{k} \in H^2$ (for $0\leq k \leq n$) we have that
\begin{equation*}\label{inq2}
\begin{aligned}
\Vert v(t_{n+1}) - v^{n+1}\Vert_{H^2} & \leq c \tau^{3/2} + \mathrm{e}^{\tau L} \Vert v(t_n)-v^n\Vert_{H^2}  \leq c \tau^{3/2} + \mathrm{e}^{\tau L} \left( c \tau^{3/2} + \mathrm{e}^{\tau L}\Vert v(t_{n-1})-v^{n-1}\Vert_{H^2}\right) \\& \leq c \tau^{3/2}\sum_{k=0}^n \mathrm{e}^{t_k L} \leq c \tau^{1/2}  t_n \mathrm{e}^{t_n L},
\end{aligned}
\end{equation*}
where $c$ depends on $\mathrm{sup}_{0 \leq t \leq t_{n+1}} \Vert v(t)\Vert_{H^3}$, $L$ depends on $\mathrm{sup}_{0 \leq k \leq n} \Vert v(t_k)\Vert_{H^3}$ as well as on $\sup_{0 \leq k \leq n} \Vert v^{k} \Vert_{H^2}$ and we have used the fact that $\hat{v}_0(t_n) \equiv \hat{v}_0^n$. The assertion then follows by a bootstrap, respectively, ``Lady Windermere's fan" argument, see, for example \cite{Faou12,HNW93,HLR12,Lubich08}.
\end{proof}

\subsection{Error analysis in $H^1$}\label{sec:gbH1} 
The error analysis  in $H^2$ of the numerical scheme \eqref{scheme} given in Section \ref{sec:h2} yields a priori bounds on the numerical solution in $H^2$ for solutions in $H^3$. This allows us to derive the following first-order convergence bound in $H^1$.
\begin{theorem}\label{thm:main}
Let the solution of \eqref{kdvT} satisfy $v(t) \in H^3$ for $t \leq T$. Then there exists a $\tau_0 > 0$ such that for all $\tau \leq \tau_0$ and $t_n \leq T$ we have
\begin{equation*}
\Vert v(t_n) - v^n \Vert_{H^1} \leq c \tau,
\end{equation*}
where $c$ depends on $\mathrm{sup}_{0 \leq t \leq t_n} \Vert v(t)\Vert_{H^3}$ and $t_n$, but can be chosen independently of $\tau$.
\end{theorem}
\begin{proof}
Note that Duhamel's formula \eqref{kdvT} implies the first-order consistency bound
\begin{equation}\label{loc10}
\begin{aligned}
\Vert \partial_x \big(\phi^\tau(v(t_k)) - \Phi^\tau_{t_k}(v(t_k)) \big)\Vert_{L^2} & \leq \int_0^\tau \Vert \partial_{x}^2 \Big[\left(\mathrm{e}^{-\partial_x^3 (t_k+s)} \phi^s(v(t_k))\right)^2 - \left(\mathrm{e}^{-\partial_x^2 (t_k+s)} v(t_k)\right)^2\Big] \Vert_{L^2}\mathrm{d}s\\
& \leq \tau c_1 \mathrm{sup}_{0 \leq t \leq \tau} \Vert \phi^t(v(t_k)) - v(t_k)\Vert_{H^2}\\
& \leq \tau^{2} c_1 \mathrm{sup}_{0\leq t \leq \tau} \Vert \phi^t(v(t_k))\Vert_{H^3},
\end{aligned}
\end{equation}
where $c_1$ depends on $ \mathrm{sup}_{0 \leq t \leq \tau} \Vert \phi^t(v(t_k))\Vert_{H^2}$.

Furthermore, as $v(t) \in H^3$ for $t \leq T$ we have the boundedness of the numerical solution in $H^2$ a priori  thanks to Theorem \ref{thm:main0}, i.e., there exists a $\tau_0>0$ such that for all $\tau \leq \tau_0$ $v^n\in H^2 $ as long as $t_n \leq T$. In particular a stability estimate of type
\begin{equation}\label{stab10}
\Vert \partial_x \big(\Phi_t^\tau(f) - \Phi_t^\tau(g)\big)\Vert_{L^2} \leq \mathrm{exp}(\tau L) \Vert \partial_x (f-g)\Vert_{L^2},\quad L = L(\Vert \partial_x^2 f \Vert_{L^2},\Vert \partial_x^3 g\Vert_{L^2})
\end{equation}
is therefore sufficient for our bootstrapping argument in $H^1$ by choosing $f= v^n \in H^2 $ and $g = v(t_n) \in H^3$. The stability bound \eqref{stab10} follows similarly to Lemma \ref{lem:stab}: Note that
\begin{equation}
\begin{aligned}\label{eq:Lf2}
& \Vert \partial_x(\Phi^\tau_t(f) - \Phi^\tau_t(g)) \Vert_{L^2}^2  = \Vert \partial_x (f - g )\Vert_{L^2}^2\\
& + \frac{1}{3} \langle \partial_x  \mathrm{e}^{\partial_{x}^3 (t+\tau)}\Big [  \left( \mathrm{e}^{-\partial_{x}^3(t+\tau)} \partial_x^{-1}f\right)^2- \left( \mathrm{e}^{-\partial_{x}^3(t+\tau)} \partial_x^{-1}g\right)^2\Big], \partial_x (f-g)\rangle\\&
 - \frac{1}{3}\langle \partial_x \mathrm{e}^{\partial_{x}^3 t} \Big[ \left( \mathrm{e}^{-\partial_{x}^3t} \partial_x^{-1}f\right)^2- \left( \mathrm{e}^{-\partial_{x}^3t} \partial_x^{-1}g\right)^2\Big], \partial_x (f-g)\rangle\\
 & + \frac{1}{6^2} \Vert 
 \partial_x  \mathrm{e}^{\partial_{x}^3 (t+\tau)}\Big [  \left( \mathrm{e}^{-\partial_{x}^3(t+\tau)} \partial_x^{-1}f\right)^2- \left( \mathrm{e}^{-\partial_{x}^3(t+\tau)} \partial_x^{-1}g\right)^2\Big]\\&
 -
 \partial_x \mathrm{e}^{\partial_{x}^3 t} \Big[ \left( \mathrm{e}^{-\partial_{x}^3t} \partial_x^{-1}f\right)^2- \left( \mathrm{e}^{-\partial_{x}^3t} \partial_x^{-1}g\right)^2\Big]\Vert_{L^2}^2\\
 & =: \Vert \partial_x (f-g)\Vert_{L^2}^2 + \frac{1}{3}I_1 + \frac{1}{6^2} I_2.
\end{aligned}
\end{equation}
Similarly to the proof of Lemma \ref{lem:I1} we can rewrite $I_1$ as
   \begin{equation}\label{I110}
   \begin{aligned}
 I_1 
  & =  \langle  \tilde{f}-\tilde{g} ,  (\tilde{f}-\tilde{g})^2 - \mathrm{e}^{\partial_{x}^3\tau }  \left( \mathrm{e}^{-\partial_{x}^3\tau }  (\tilde f-\tilde g)\right)^2\rangle
\\
&+  \langle \tilde{g} , (\tilde{f}-\tilde{g})^2- \mathrm{e}^{\partial_{x}^3\tau }  \left( \mathrm{e}^{-\partial_{x}^3\tau }(\tilde{f}-\tilde{g})\right)^2\rangle
\\&
- 2 \langle   \tilde f-\tilde g,  \tilde{g} (\tilde{f}-\tilde{g}) - \mathrm{e}^{\partial_{x}^3\tau }\Big[\left( \mathrm{e}^{-\partial_{x}^3\tau } \tilde{g}\right)\left( \mathrm{e}^{-\partial_{x}^3\tau }(\tilde{f}-\tilde{g})\right)\Big] \rangle
\\&
- 2 \langle   \tilde f-\tilde g, ( \partial_x \tilde{g}) \left(\partial_x^{-1} (\tilde{f}-\tilde{g}) \right)- \mathrm{e}^{\partial_{x}^3\tau }\Big[\left( \mathrm{e}^{-\partial_{x}^3\tau }(\partial_x \tilde{g})\right)\left( \mathrm{e}^{-\partial_{x}^3\tau }\partial_x^{-1}(\tilde{f}-\tilde{g})\right)\Big] \rangle.
\end{aligned}
\end{equation}
As in Lemma \ref{lem:kb} we obtain by the key relation \eqref{eq:key} that
\begin{equation*}
I(u,v,w):= \vert \langle u  , v w - \mathrm{e}^{\partial_x^3 \tau} \Big[ \left(\mathrm{e}^{-\partial_x^3 \tau} v \right)  \left(\mathrm{e}^{-\partial_x^3 \tau}w \right)\Big] \rangle \vert
 \leq 3 \tau  \sum_{k_1,k_2} \vert (k_1+k_2) \hat{u}_{-(k_1+k_2)}\vert  \vert  k_1k_2\vert \vert \hat v_{k_1} \hat{w}_{k_2}\vert.
\end{equation*}
The Cauchy-Schwarz and Young inequality furthermore yield that
\begin{equation}\label{eb1}
\begin{aligned}
I(u,v,w) & \leq 3 \tau \sum_{k,l} |l | \vert \hat{u}_{-l}\vert \vert (l-k) k\vert \vert \hat{v}_k \hat{w}_{l-k}\vert \leq 3 \tau \big( \sum_l \vert l\vert^2 \vert \hat{u}_l\vert^2\big)^{1/2} \Vert v^{(1)} \ast w^{(1)}\Vert_{l^2}\\
& \leq c \tau \Vert \partial_x u \Vert_{L^2} \mathrm{min}\left(\Vert v^{(1)} \Vert_{l^1} \Vert w^{(1)} \Vert_{l^2},\Vert v^{(1)} \Vert_{l^2} \Vert w^{(1)} \Vert_{l^1} \right)\\
& \leq c \tau \Vert \partial_x u \Vert_{L^2} \mathrm{min}\left(\Vert \partial_x^2 v \Vert_{L^2} \Vert \partial_x w \Vert_{L^2},\Vert \partial_x^2 w \Vert_{L^2} \Vert \partial_x v\Vert_{L^2} \right).
\end{aligned}
\end{equation}
The above bound allows us to control the first and last two terms in \eqref{I110} as long as $f-g,g \in H^2$. Furthermore,
\begin{equation}\label{eb2}
\begin{aligned}
I(u,v,v) & \leq 3 \tau \sum_{k,l} |l | \vert \hat{u}_{-l}\vert \vert (l-k) k\vert \vert \hat{v}_k \hat{v}_{l-k}\vert \leq 3\tau  \big( \sum_k \vert k\vert^2 \vert \hat{v}_k\vert^2\big)^{1/2} \Vert u^{(1)} \ast v^{(1)}\Vert_{l^2}\\
& \leq 3 \tau \Vert \partial_x v\Vert_{L^2} \Vert u^{(1)}\Vert_{l^1} \Vert v^{(1)}\Vert_{l^2} \leq c \tau \Vert \partial_x v \Vert_{L^2}^2 \Vert \partial_x^2 u \Vert_{L^2},
\end{aligned}
\end{equation}
which allows us to control the second term in \eqref{I110} as long as $f-g\in H^1$ and $g \in H^2$. 

Using the bounds \eqref{eb1} and \eqref{eb2} in \eqref{I110} yields that
\begin{equation}\label{boundI12}
\vert I_1\vert \leq  \tau L \Vert \partial_x (f-g)\Vert_{L^2}^2,\quad L = L(\Vert \partial_x^2 (f-g)\Vert_{L^2},\Vert \partial_x^2 g \Vert_{L^2}). 
\end{equation}
Next we write $I_2 = I_2^a + I_2^b$ with
 \begin{equation*}
 \begin{aligned}
 I_2^a  =& \langle  
  \partial_x   \left( \mathrm{e}^{-\partial_{x}^3\tau } \partial_x^{-1}\tilde{f}\right)^2- \partial_x\left( \mathrm{e}^{-\partial_{x}^3\tau } \partial_x^{-1}\tilde{g}\right)^2- \partial_x \mathrm{e}^{-\partial_{x}^3 \tau} \Big[ \left(\partial_x^{-1}\tilde{f}\right)^2- \left(  \partial_x^{-1}\tilde{g}\right)^2\Big]
,\\&\quad
   \partial_x   \left( \mathrm{e}^{-\partial_{x}^3 \tau } \partial_x^{-1}\tilde{f}\right)^2-\partial_x \left( \mathrm{e}^{-\partial_{x}^3\tau} \partial_x^{-1}\tilde{g}\right)^2 \rangle\\
  I_2^b  =&  -  \langle \partial_x   \mathrm{e}^{\partial_{x}^3 \tau} \Big[ \left( \mathrm{e}^{-\partial_{x}^3 \tau} \partial_x^{-1} \tilde{f}\right)^2- \left( \mathrm{e}^{-\partial_{x}^3\tau} \partial_x^{-1}\tilde{g}\right)^2 \Big] - \partial_x  \left( \partial_x^{-1} \tilde{f} \right)^2- \partial_x\left( \partial_x^{-1}\tilde{g}\right)^2,\\&
\quad \partial_x \left(\partial_x^{-1}\tilde{f}\right)^2- \partial_x \left(  \partial_x^{-1}\tilde{g}\right)^2
 \rangle.
 \end{aligned}
 \end{equation*}
 Note that by the Cauchy-Schwarz and Young inequality we have with $F: = \partial_x \left(\partial_x^{-1}\tilde{f}\right)^2- \partial_x  \left(\partial_x^{-1}\tilde{g}\right)^2$ that
  \begin{equation}\label{I2bb}
 \begin{aligned}
  \vert I_2^b \vert & = \vert \sum_{k_1,k_2} \hat{F}_{-(k_1+k_2)} \frac{(k_1+k_2) }{k_1 k_2} \left(1 - \mathrm{e}^{-i \tau \big((k_1+k_2)^3 - k_1^3-k_2^3\big)} \right)\left( \hat{\tilde{f}}_{k_1} \hat{\tilde{f}}_{k_2} - \hat{\tilde{g}}_{k_1} \hat{\tilde{g}}_{k_2}\right)\vert \\
 &  = \vert \sum_{k_1,k_2}\hat{F}_{-(k_1+k_2)} \frac{(k_1+k_2) }{k_1 k_2} \left(1 - \mathrm{e}^{-i \tau 3 k_1k_2 (k_1+k_2)} \right)\left( \hat{\tilde{f}}_{k_1} \hat{\tilde{f}}_{k_2} - \hat{\tilde{g}}_{k_1} \hat{\tilde{g}}_{k_2}\right)\vert \\
 & \leq 3\tau   \sum_{k_1,k_2} \vert \hat{F}_{-(k_1+k_2)}(k_1+k_2)\vert  \vert (k_1+k_2)\vert \left( \hat{\tilde{f}}_{k_1} \hat{\tilde{f}}_{k_2} - \hat{\tilde{g}}_{k_1} \hat{\tilde{g}}_{k_2}\right)\vert\\
 & \leq 3 \tau \Big( \sum_l l^2 \vert \hat{F}_{l}\vert^2\Big)^{1/2} \Big( \Vert (\tilde{f}-\tilde{g})^{(1)} \ast \tilde f^{(0)}\Vert_{l^2} +  \Vert (\tilde{f}-\tilde{g})^{(0)} \ast \tilde f^{(1)}\Vert_{l^2} \\
 &+3 \tau \Big( \sum_l l^2 \vert \hat{F}_{l}\vert^2\Big)^{1/2} \Big( \Vert (\tilde{f}-\tilde{g})^{(1)} \ast \tilde g^{(0)}\Vert_{l^2} +  \Vert (\tilde{f}-\tilde{g})^{(0)} \ast \tilde g^{(1)}\Vert_{l^2} 
 \Big)\\
 & \leq c \tau \Vert \partial_x F\Vert_{L^2} \left(\Vert (\tilde{f}-\tilde{g})^{(1)}\Vert_{l^2} \big(\Vert \tilde{f}^{(0)}\Vert_{l^1}+ \Vert \tilde{g}^{(0)} \Vert_{l^1}\big)+\Vert (\tilde{f}-\tilde{g})^{(0)}\Vert_{l^1} \big(\Vert \tilde{f}^{(1)}\Vert_{l^2}+ \Vert \tilde{g}^{(1)} \Vert_{l^2}\big)\right) \\
 & \leq M \tau  \Vert \partial_x (\tilde{f}-\tilde{g})\Vert_{L^2}^2,
 \end{aligned}
 \end{equation}
where $M$ depends on $\Vert \partial_x f\Vert_{L^2}$ and $\Vert \partial_x g\Vert_{L^2}$. A similar bound holds for $I_2^a$ which implies that
\begin{equation}\label{boundI210}
\vert I_2 \vert \leq M \tau  \Vert \partial_x (f-g)\Vert_{L^2}^2,\quad M = M(\Vert \partial_x f\Vert_{L^2},\Vert \partial_x g\Vert_{L^2}).
\end{equation}
Plugging the bounds \eqref{boundI12} as well as \eqref{boundI210} into \eqref{eq:Lf2} yields the stability estimate \eqref{stab10}.

With the aid of the stability estimate \eqref{stab10} and the local error bound \eqref{loc10} the proof then follows the line of argumentation to the proof of Theorem \ref{thm:main0}.
\end{proof}
\begin{corollary}\label{cor:ucon}
Let the solution of the KdV equation \eqref{eq:kdv} satisfy $u(t) \in H^3$ for $t \leq T$. Then there exists a $\tau_0 > 0$ such that for all $\tau \leq \tau_0$ and $t_n \leq T$ the exponential-type integration scheme~\eqref{schemeU} is first-order convergent in $H^1$, i.e.,
\begin{equation*}
\Vert u(t_n) - u^n \Vert_{H^1} \leq c \tau,
\end{equation*}
where $c$ depends on $\mathrm{sup}_{0 \leq t \leq t_n} \Vert u(t)\Vert_{H^3}$ and $t_n$, but can be chosen independently of $\tau$.
\end{corollary}
\begin{proof}
The assertion follows from Theorem \ref{thm:main} as $\mathrm{e}^{t\partial_x^3 }$ is a linear isometry in $H^1$ for all $t\in \mathbb{R}$.
\end{proof}
\section{Numerical experiments}\label{sec:num}
 In this section, we numerically underline the first- and second-order convergence rates of the exponential-type integration schemes~\eqref{schemeU} and \eqref{2scheme}, respectively, towards the exact solution of the KdV equation \eqref{eq:kdv}.  For the space discretization we use a Fourier pseudo spectral method, see \cite{YMQ88}, where we choose the largest Fourier mode $K = 2^{12}$. Details on the fully discrete scheme are given in Remark \ref{rem:AFPSE}.
\begin{remark}\label{rem:AFPSE} We employ the following fully discrete Fourier pseudo spectral version of~\eqref{schemeU}: Set $B^K = \{-K/2,\ldots,K/2-1\}$ and let $\mathcal{F}_K: B^K \rightarrow B^K$ denote the discrete Fourier transform and $\mathcal{F}_K^{-1}$ its inverse. Denote by $ u^{K,0}$ the discretized  initial value vector on the grid  $x_a = \frac{2\pi}{K} a$, $a \in B^K$ and set
\[
 \xi^{K,0} = \mathcal{F}_K \left(u^{K,0}\right) = \left[ (\mathcal{F}_K u^{K,0})_{-\frac{K}{2}},(\mathcal{F}_K u^{K,0})_{-\frac{K}{2}+1}, \ldots,  (\mathcal{F}_K u^{K,0})_{\frac{K}{2}-1}
\right].
\]
With this notation at hand a fully discrete Fourier pseudo spectral version of \eqref{schemeU} reads
\begin{equation*}
\begin{aligned}\label{FPSEI}
 & \xi^{K,n+1} =  \mathrm{e}^{- \tau\widehat \partial_{x,K}^3 } \xi^{K,n} + \frac{1}{6}\mathcal{F}_K \left(
\left[ \mathcal{F}_K^{-1}\left(\mathrm{e}^{- \tau\widehat \partial_{x,K}^3 }\widehat \partial_{x,K}^{-1}  \xi ^{K,n}\right)\right]^2\right)
-\frac{1}{6} \mathrm{e}^{- \tau \widehat \partial_{x,K}^3 } \mathcal{F}_K \left( \left[ \mathcal{F}_K^{-1}\left( \widehat \partial_{x,K}^{-1} \xi ^{K,n}\right)\right]^2\right)
\end{aligned}
\end{equation*}
with $ u^{K,n+1} = \mathcal{F}_K^{-1}  \left(\xi^{K,n+1}\right)$. Thereby, the multiplication of two vectors is taken point-wise, i.e.,
\[
 \left[x_{1},\ldots, x_{K}\right]\left[y_{1}, \ldots, y_{K}\right]= \left[x_{1} y_{1}, \ldots, x_{K}y_{K}\right]
\]
and the discrete differential operators acting in Fourier space are defined through
\begin{align*}
&\widehat \partial_{x,K} := \textstyle i \big[-K/2,  \ldots,  K/2-1\big],\qquad \widehat \partial_{x,K}^{-1} := \textstyle \frac{1}{i} \left[\frac{1}{-K/2},  \ldots, -1,0,1,\ldots \frac{1}{K/2-1}\right].
\end{align*}
\end{remark}
\begin{example} In the first numerical experiment we choose the initial value
\begin{equation}\label{In}
u(0,x) = 2\, \mathrm{sech}^2 \left(\textstyle \frac{x}{2}\right) \mathrm{sin}(x)\qquad \text{with} \quad  \mathrm{sech}(x) = \frac{1}{\mathrm{cosh}(x)}
\end{equation}
and integrate the  exponential-type integration schemes~\eqref{schemeU} and \eqref{2scheme}   up to $T=2$. As the exact solution is unknown we take as a reference solution the second-order scheme itself with a very small time-step size $\tau = 10^{-7}$.  The error between the numerical solutions and the reference solution at time $T=2$ as well as a graph of the initial value, the reference solution and the  first-order approximate solution   is given in Figure~\ref{fig:kdv}.


\begin{figure}[h!]
\centering
   \begin{subfigure}[b]{0.4\textwidth}
        \includegraphics[width=\textwidth]{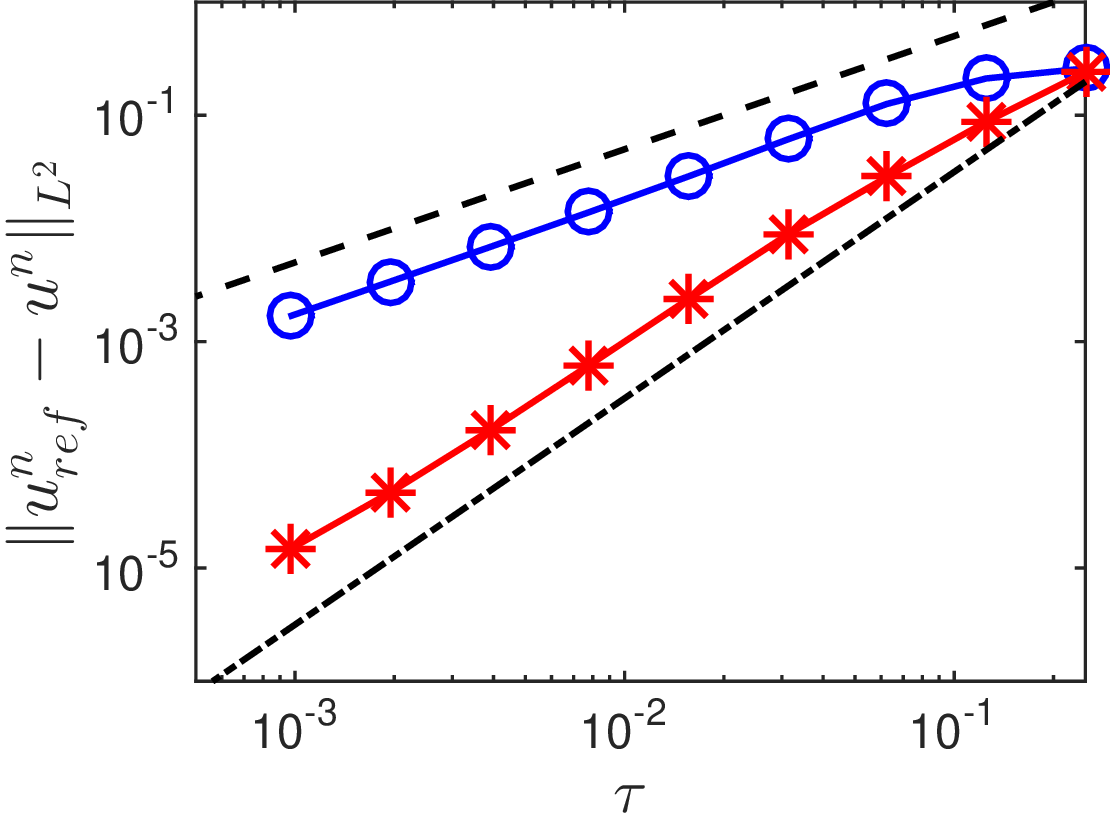}
        \caption{Orderplot (double logarithmic).  Convergence rates of the first-order scheme \eqref{schemeU} (blue, circle) and the second-order scheme \eqref{2scheme} (red, star). The slopes of the dashed and dashed-dotted lines are one and two, respectively.}
    \end{subfigure}\hfill
       \begin{subfigure}[b]{0.44\textwidth}
        \includegraphics[width=\textwidth]{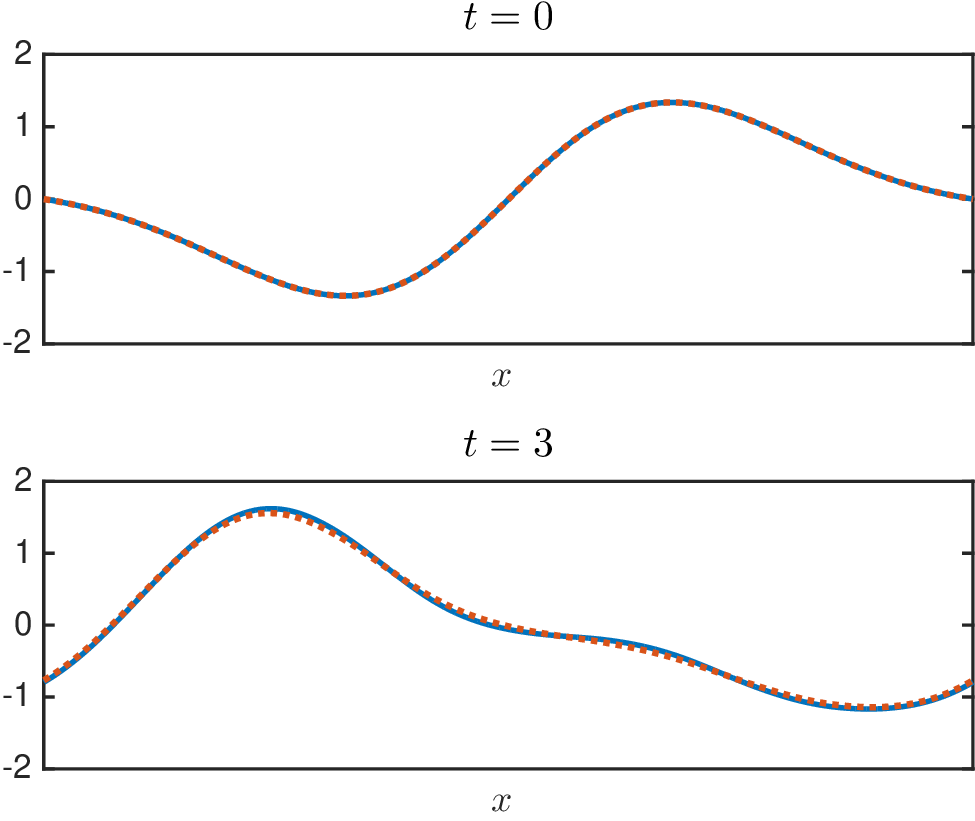}
        \caption{Time evolution of the reference solution $u_{ref}^n$ (orange, dotted) and the  first-order approximate solution  $u^n$ with $\tau \approx 10^{-2}$ (blue, continuous).}
    \end{subfigure}
\caption{(Initial value \eqref{In})  Numerical simulation of the  first- and second-order  exponential-type integration schemes \eqref{schemeU} and \eqref{2scheme}.}\label{fig:kdv}
\end{figure}


\end{example}

\begin{example}[Solitary waves] The KdV equation
\[
\partial_t \phi + \partial_x^3 \phi + \frac{1}{2} \partial_x (\phi^2) = 0, \qquad x \in \mathbb{R}
\]
allows solitary wave solutions of type
\begin{equation}\label{solW}
\phi(t,x) =  3 c \, \mathrm{sech}^2\left( \frac{\sqrt{c}}{2}(x-c t - a)\right) \qquad \text{with} \quad  a\in \mathbb{R}, \,c>0.
\end{equation}
In order to test the resolution of solitary waves under the schemes \eqref{schemeU} and \eqref{2scheme} we choose a ``large torus'' $\mathbb{T}_L = [-\frac{\pi}{L}, \frac{\pi}{L}]$ with $L=0.1$ such that boundary errors are negligible. Furthermore, we fix $c = 1$ and $a = 0$.  The $H^1$-error between the first- and second-order exponential-type integration schemes~\eqref{schemeU} and \eqref{2scheme}, respectively, and the exact solution \eqref{solW}  at time $T=2$ is illustrated in Figure \ref{fig:kdvSolOrd}. 
\begin{figure}[h!]
\centering
\includegraphics[width=0.4\linewidth]{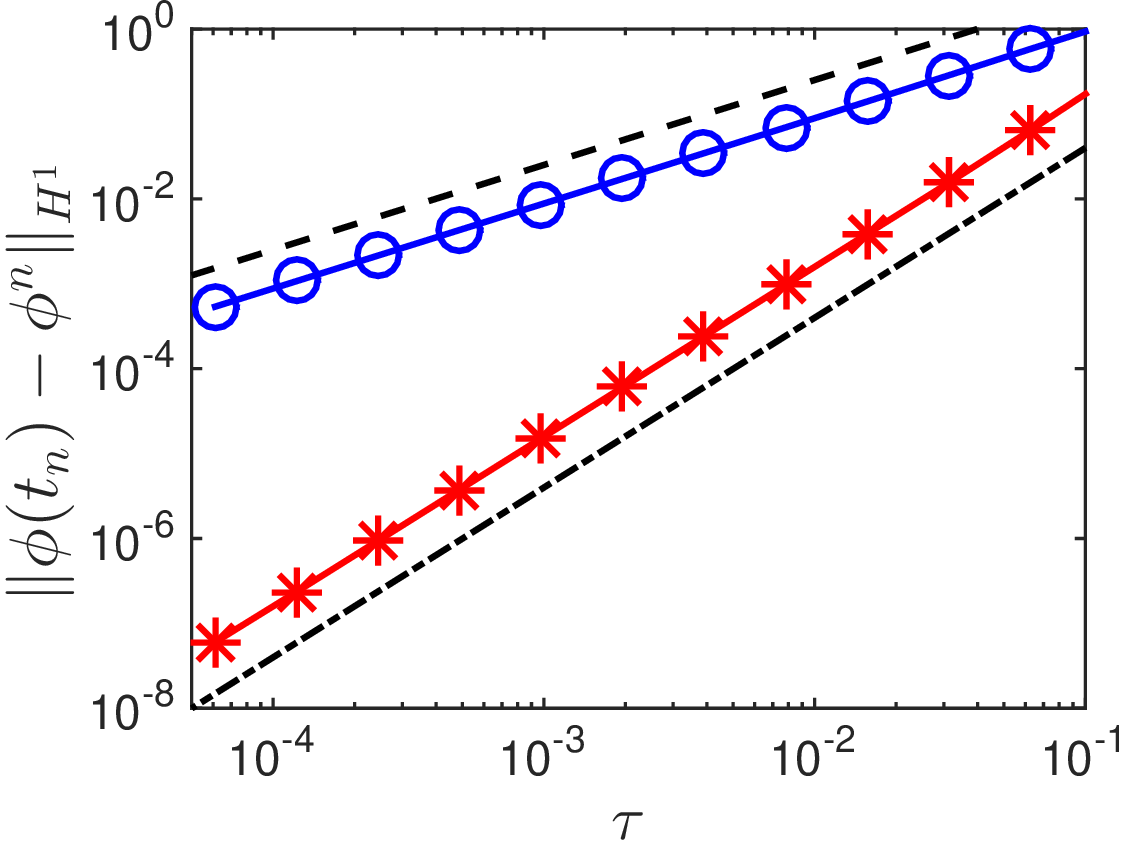}
\caption{(Solitary wave) Orderplot (double logarithmic).  Convergence rates of the first-order scheme \eqref{schemeU} (blue, circle) and the second-order scheme~\eqref{2scheme} (red, star) measured in a discrete $H^1$ norm. The slopes of the dashed and dashed-dotted lines are one and two, respectively.}\label{fig:kdvSolOrd}
\end{figure}

  A graph of the time evolution of the solitary wave solution  \eqref{solW} (with $c = 1.2$, $a = -5\pi$) and the corresponding first- and second-order approximate solutions~\eqref{vtS} and \eqref{2scheme}, respectively, for two different time step sizes is illustrated in Figure~\ref{fig:SolEv}.

\begin{figure}[h!]
\centering
   \begin{subfigure}[b]{0.47\textwidth}
        \includegraphics[width=\textwidth]{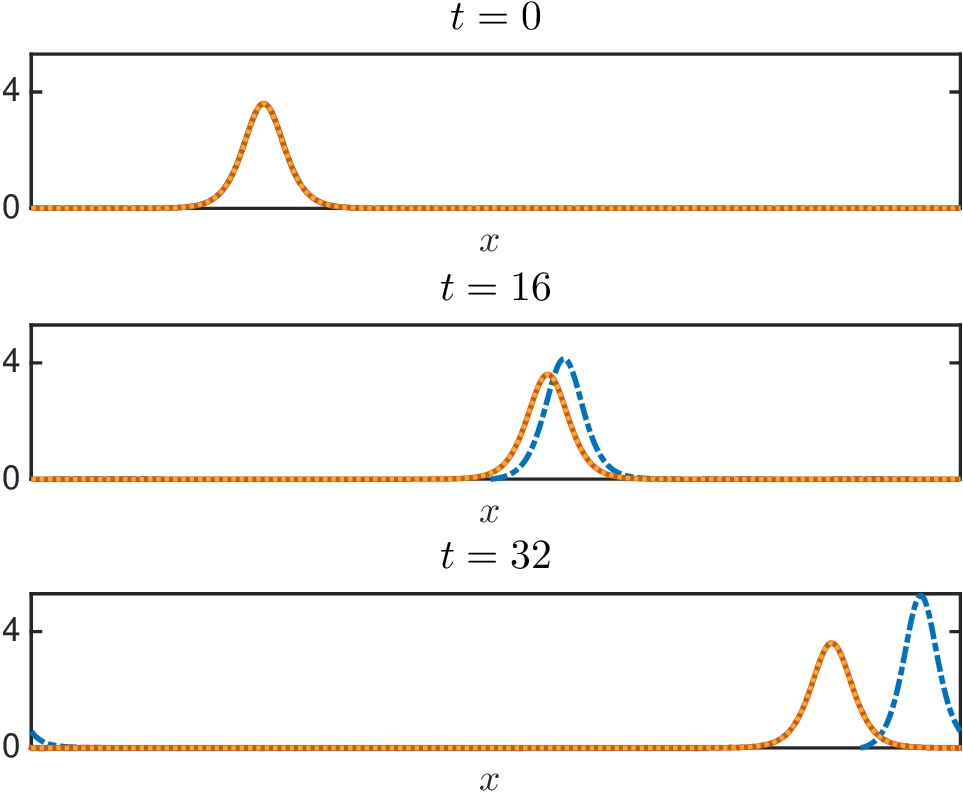}
        \caption{$\tau = 10^{-2}$}
        \label{fig:mouse}
    \end{subfigure}\hfill
       \begin{subfigure}[b]{0.47\textwidth}
        \includegraphics[width=\textwidth]{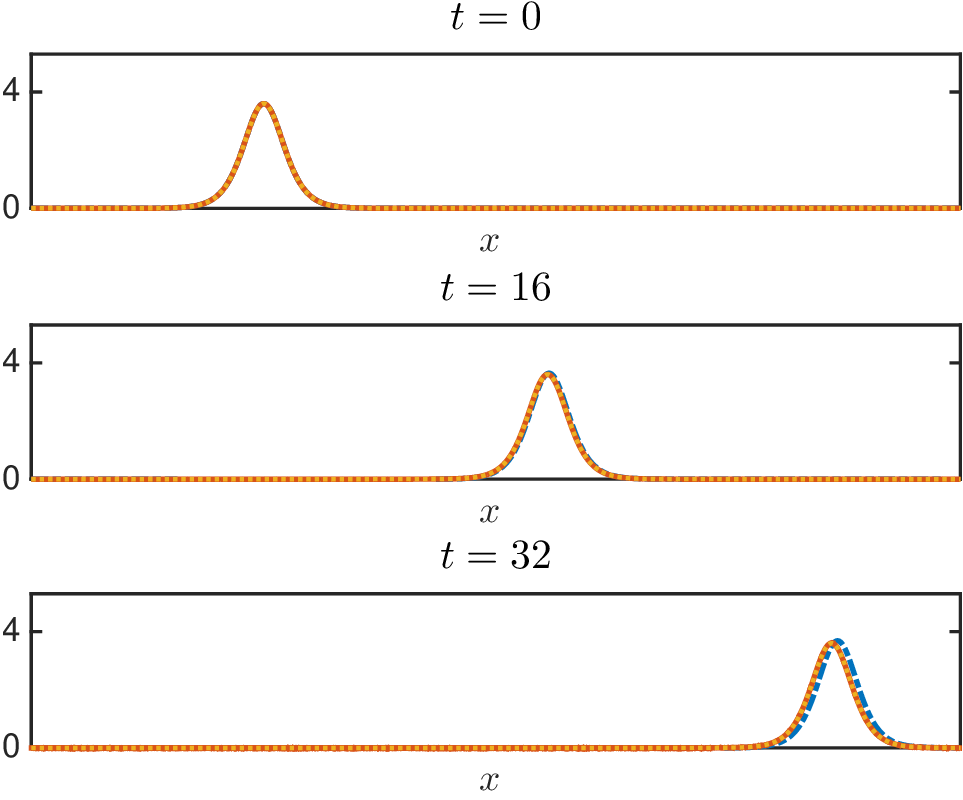}
        \caption{$\tau = 10^{-3}$}
        \label{fig:mou}
    \end{subfigure}
\caption{ Time evolution of the solitary wave \eqref{solW} (yellow, dotted),  the first-order approximate solution~\eqref{vtS}  (blue, dashed-dotted) and second-order approximate solution \eqref{2scheme} (red, continous) for two different time-step sizes $\tau$.}\label{fig:SolEv}
\end{figure}

\end{example}

\section*{Acknowledgement}
K. Schratz gratefully acknowledges financial support by the Deutsche Forschungsgemeinschaft (DFG) through CRC 1173.


\begin{thebibliography}{}

\bibitem{Bour93}
{\rm J. Bourgain},
{\em Fourier transform restriction phenomena for certain lattice subsets and applications to nonlinear evolution equations. Part II: The KdV-equation.} Geom. Funct. Anal. 3:209--262 (1993).

\bibitem{OE15}
{\rm L. Einkemmer, A. Ostermann},
{\em A splitting approach for the Kadomtsev-Petviashvili equation.} J. Comput. Phys. 299:716--730 (2015).

\bibitem{Faou12}
{\rm E. Faou},
{\em Geometric numerical integration and Schr\"odinger equations.} European Math. Soc (2012).

\bibitem{Gau11}
{\rm L. Gauckler},
{\em Convergence of a split-step Hermite method for the Gross-Pitaevskii equation.}
IMA J. Numer. Anal. 31:1082--1106 (2011).


\bibitem{Gub11}
{\rm M. Gubinelli},
{\em Rough solutions for the periodic Korteweg-de Vries equation.} Comm. Pure Appl. Anal. 11:709--733 (2012).

\bibitem{HNW93}
{\rm E. Hairer, S. P. N\o rsett, G. Wanner},
{\em Solving Ordinary Differential Equations I. Nonstiff
Problems.} Second edition. Springer, Berlin, (1993).

\bibitem{HLW}
{\rm E. Hairer, C. Lubich, G. Wanner},
{\em Geometric Numerical Integration. Structure-Preserving Algorithms for Ordinary Differential Equations.} Second Edition, Springer (2006).

\bibitem{HochOst10}
{\rm M. Hochbruck, A. Ostermann},
{\em Exponential integrators.}  Acta Numer. 19:209--286 (2010).

\bibitem{HLRS10}
{\rm H. Holden, K. H. Karlsen, K.-A. Lie, N. H. Risebro},
{\em Splitting for Partial Differential Equations with Rough Solutions.} European Math. Soc. Publishing House, Z\"urich (2010).

\bibitem{HLR12}
{\rm H. Holden, C. Lubich, N. H. Risebro},
{\em Operator splitting for partial differential equations with Burgers nonlinearity.} Math. Comp. 82:173--185 (2012).

\bibitem{HKRT12}
{\rm H. Holden, K. H. Karlsen, N. H. Risebro, T. Tao},
{\em Operator splitting methods for the Korteweg-de Vries equation.} Math. Comp. 80:821--846 (2011).

\bibitem{HKR99}
{\rm H. Holden, K. H. Karlsen, N. H. Risebro},
{\em Operator splitting methods for generalized Korteweg-de Vries equations}. J. Comput. Phys. 153:203--222 (1999).


\bibitem{KT05}
{\rm A-K Kassam, L. N. Trefethen},
{\em Fourth-order time-stepping for stiff PDEs.} SIAM J. Sci. Comput. 26:1214--1233 (2005).

\bibitem{Klein06}
{\rm C. Klein},
{\em Fourth order time-stepping for low dispersion Korteweg-de Vries and nonlinear Schr\"odinger equation.} ETNA 29:116--135 (2008).

\bibitem{Law67}
{\rm J. D. Lawson},
{\em Generalized Runge-Kutta processes for stable systems with large Lipschitz constants.} SIAM J. Numer. Anal. 4:372--380 (1967).

\bibitem{Lubich08}
{\rm C. Lubich}, 
{\em On splitting methods for {S}chr\"odinger-{P}oisson and cubic nonlinear {S}chr\"odinger
equations.} Math. Comp. 77:2141--2153 (2008).

\bibitem{YMQ88}
{\rm Y. Maday, A. Quarteroni},
{\em Error analysis for spectral approximation of the Korteweg-de Vries equation.} RAIRO - Mod\'elisation math\'ematique et analyse num\'erique 22:821--846 (1988).

\bibitem{McLacQ02}
{\rm R.I. McLachlan, G.R.W. Quispel.}
{\em Splitting methods},
Acta Numer. 11:341--434 (2002).

\bibitem{Tao06}
{\rm T. Tao},
{\em Nonlinear Dispersive Equations. Local and Global Analysis.} Amer. Math. Soc. Providence (2006).

\bibitem{T74}
{\rm F. Tappert},
{\em Numerical solutions of the Korteweg-de Vries equation and its generalizations
by the split-step Fourier method.} In: (A. C. Newell, editor) Nonlinear Wave Motion, Amer.
Math. Soc. 215--216 (1974).


\end{thebibliography}
\end{document}